\documentclass{amsart}
\usepackage{amsmath}
\usepackage{pdfpages}
\usepackage{makecell}
\usepackage{amsxtra}
\usepackage{indentfirst}
\usepackage{algorithm}
\usepackage{algpseudocode}
\usepackage{ytableau}
\usepackage[english]{babel}
\usepackage{tikz-cd}
\usepackage{amssymb,amscd}
\usepackage[sort]{natbib}
\usepackage{amsthm,graphicx}
\usepackage{mathtools}
\usepackage{breqn}
\usepackage{enumerate}
\usepackage[shortlabels]{enumitem}
\usepackage{amsthm}
\makeatletter
\renewenvironment{proof}[1][\proofname] {\par\pushQED{\qed}\normalfont\topsep6\p@\@plus6\p@\relax\trivlist\item[\hskip\labelsep\bfseries#1\@addpunct{.}]\ignorespaces}{\popQED\endtrivlist\@endpefalse}
\makeatother
\usepackage{appendix}
\usepackage{tikz}
\usepackage{mathrsfs}
\usepackage{hyperref}
\usepackage{pgfplots}
\pgfplotsset{compat=1.15}
\usepackage{mathrsfs}
\usetikzlibrary{arrows}
\usepackage[english]{babel}
\usepackage{geometry}
\usepackage{thmtools}
\usepackage{cleveref}
\usepackage{thm-restate}
\newtheorem{theorem}{Theorem}[section]
\newcommand{\thm}{\begin{theorem}}
\newcommand{\et}{\end{theorem}}
\newtheorem{lemma}[theorem]{Lemma}

\newtheorem{question}[theorem]{Question}
\newtheorem{definition}[theorem]{Definition}
\newtheorem{proposition}[theorem]{Proposition}
\newtheorem{conjecture}[theorem]{Conjecture}
\newtheorem{consequence}[theorem]{Corollary}
\newtheorem{eg}{Example}[theorem]
\newtheorem{note}[theorem]{Remark}
\newcommand{\bl}{\begin{lemma}}
\newcommand{\el}{\end{lemma}}
\newcommand{\defi}{\begin{definition}}
\newcommand{\ed}{\end{definition}}
\newcommand{\cor}{\begin{consequence}}
\newcommand{\ec}{\end{consequence}}
\newcommand{\egs}{\begin{eg}}
\newcommand{\ee}{\end{eg}}
\newcommand{\qs}{\begin{question}}
\newcommand{\eqs}{\end{question}}
\begin{document}
\crefname{theorem}{Theorem}{Theorem}
\crefname{chapter}{Chapter}{Chapter}
\crefname{problem}{Problem}{Problem}
\crefname{lemma}{Lemma}{Lemma}
\crefname{consequence}{Corollary}{Corollary}
\crefname{definition}{Definition}{Definition}
\crefname{case}{Case}{Case}
\crefname{proposition}{Proposition}{Proposition}
\crefname{property}{Property}{Property}
\crefname{conjecture}{Conjecture}{Conjecture}
\crefname{section}{Section}{}
\crefname{question}{Question}{Question}
\crefname{note}{Remark}{Remark}
\crefname{algorithm}{Algorithm}{Algorithm}
\crefname{figure}{Figure}{Figures}
\title{Surface subgroups of Baumslag doubles along short words}
\author[H. Le]{Le Xuan Hoang}
\address{VNU University of Science}
\email{lexuanhoang_t66@hus.edu.vn}
\author[H. Tran]{Tran Nguyen Nam Hung}
\address{Ho Chi Minh City University of Science}
\email{nguyennamtranhung1303@gmail.com}
\date{\today}
\begin{abstract}
    If $U$ is a minimal, diskbusting, finite list of words in a free group $F_n$ of rank $n$ such that the sum of the lengths of words in $U$ is at most $2n+4$, we prove that the natural presentation complex of the Baumslag double of $F_n$ along $U$ virtually contains a $\pi_1$-injective embedded closed hyperbolic surface. This verifies the Tiling Conjecture of Kim and Wilton for this type of lists of words, and in particular, implies that the corresponding Baumslag double contains a hyperbolic surface subgroup.  
\end{abstract}
\maketitle
\setcounter{section}{0}
\section{Introduction}
A \textit{hyperbolic surface group} is the fundamental group of a closed, orientable $2$-manifold with negative Euler characteristic. Denote by $F_n$ the free group of rank $n$ with a fixed basis $\mathcal{A}_n = \{a_1, \ldots, a_n\}$. Let $U$ be a list of words in $F_n$, a \textit{double of a free group} is the fundamental group of a graph of spaces $X(U)$ having two vertex spaces homeomorphic to $\bigvee\limits_{i=1}^{n}S^{1}$; furthermore, a cylindrical edge space is glued along the two copies of each word in $U$. Denote by $D(U):= \pi_1(X(U))$ the double of $F_n$ with respect to the list $U$. We also call $D(U)$ as the \emph{Baumslag double} of $F_n$ along $U$.\smallbreak
Our initial purpose is to study the existence of a $\pi_1$-injective immersion of a hyperbolic surface in $X(U)$.
\begin{question}
    When does $X(U)$ virtually contain a $\pi_1$-injective closed hyperbolic surface?
\end{question}
The above question is motivated by the following question by M. Gromov.  
\begin{question}
    [\cite{gromov1992asymptotic}]\label{gromov_q} Does every one-ended word-hyperbolic group have a hyperbolic surface subgroup?
\end{question}
\cref{gromov_q} has been answered affirmatively for the following cases.
\begin{enumerate}
    \item Coxeter groups \cite{GORDON2004135}.
    \item Fundamental groups of closed hyperbolic $3$-manifolds \cite{e6b728f3-659c-39d6-90a1-237df1230f55}.
    \item Graphs of free groups with infinite cyclic edge groups \cite{calegari_2023_rcpw2-ejn35, wilton2017}.
\end{enumerate}
An interesting case is when the group is given as a Baumslag double of a free group. In particular, in \cite{kim2010geometricity, Kim_2010_wilton}, the authors formulated a combinatorial group-theoretic condition on the list of words $U$, which is equivalent to the condition that the natural presentation complex for $D(U)$ virtually contains a $\pi_1$-injective, embedded, closed hyperbolic surface. They conjectured that this condition always holds when $D(U)$ is one-ended, word-hyperbolic, and $U$ is minimal among its automorphic copies in $F_n$. This conjecture is called the Tiling Conjecture. Kim and Oum \cite{kim2010} verified the Tiling Conjecture for $D(U)$ if (1) the free group has rank two, or (2) every generator is used the same number of times in a minimal automorphic image of the amalgamating words.\smallbreak
Let $U$ be a finite list of words in $F_n$, by the word “list” we mean that repetitions are allowed. We say that $U$ is \emph{minimal} if no automorphism of $F_n$ reduces the sum of the lengths of the words in $U$. Our main result shows that Tiling Conjecture is true if the sum of the length of words in $U$ is at most $2n+4$.
\thm\label{main_thm}
Let $U$ be a minimal list of words in $F_n$. If $D(U)$ is one-ended, word-hyperbolic, and the sum of the lengths of the words in $U$ is at most $2n+4$, then the natural presentation complex for $D(U)$ virtually contains a $\pi_1$-injective, embedded closed hyperbolic surface.
\et
In \cite{wilton2017}, Wilton proved that there exists a hyperbolic surface subgroup inside the fundamental group of each one-ended word-hyperbolic graph of free groups with cyclic edge groups. From \cref{main_thm}, we obtain an alternative proof of his result in the special case when the group in consideration is $D(U)$.
\cor\label{like_wilton}
If the double $D(U)$ of a rank-$n$ free group is one-ended and word-hyperbolic, and if the sum of the lengths of the words in $U$ is at most $2n+4$, then $D(U)$ contains a hyperbolic surface subgroup.
\ec
In order to enhance readability, we briefly explain the strategy of proving \cref{main_thm}. First of all, we may always assume that the list $U$ is minimal, moreover, the condition that $D(U)$ is one-ended is equivalent to the condition that the list $U$ is \textit{diskbusting} \cite{Gordon_2010}. These combinatorial conditions on the list of words allow one to formulate a graph-theoretic conjecture on the existence of hyperbolic surfaces in the natural presentation complex for $D(U)$, the Tiling Conjecture \cite{kim2010, kim2010geometricity, Kim_2010_wilton}, which we recall in \cref{prelims}. In \cref{main_proof_section}, we prove a special case of this conjecture by an inductive argument involving subdivisions of graphs, and then deduce \cref{main_thm}. The base step of the induction is proven using results in the case of doubles of free groups of rank two and the case of regular lists by Kim and Oum \cite{kim2010}. In \cref{discussion}, we give computational evidences for a stronger conjecture than the Tiling Conjecture. 
\subsection{Acknowledgement}
We would like to thank Professor Sang-hyun Kim for providing guidance and support throughout the project. We also would like to thank the organizers of Vietnam-Polymath-REU for creating an opportunity for us to enhance our skills and knowledge by collaboratively working on research problems.
\subsection{AI disclosure statement}
We do not use AI in any part of the project. 
\section{Preliminary}\label{prelims}
We explain the connection between Tiling Conjecture \cite{kim2010, kim2010geometricity, Kim_2010_wilton} and Gromov's question on hyperbolic surface subgroups of one-ended word-hyperbolic groups. Much of the material in this section can be found in \cite{kim2010}.\smallbreak
\subsection{Baumslag doubles of free groups, Whitehead graphs, and polygonality}
\emph{Baumslag doubles} of free groups are defined as follows. Given a natural number $n\ge 2$, let $F_n$ be the free group on a set $\mathcal{A}_n$ having $n$ elements. Let $U = [u_1, \ldots, u_r]$ be a list of nontrivial words in $F_n$. Suppose that $F^{(1)}_n$ and $F^{(2)}_n$ are two copies of $F_n$. Denote by $X(U)$ the graph of spaces with two vertex spaces, which are $\bigvee\limits_{i=1}^{n}S^{1}_{(1)} = \operatorname{Cay}(F^{(1)}_n)/F^{(1)}_n$ and $\bigvee\limits_{i=1}^{n}S^{1}_{(2)} = \operatorname{Cay}(F^{(2)}_n)/F^{(2)}_n$, here $\operatorname{Cay}(F_n)$ denotes the Cayley graph of a rank-$n$ free group with its standard generating set, on which the group $F_n$ acts naturally; $X(U)$ has $r$ edge spaces $e_1, \ldots, e_r$ connecting the two vertices. For $i\in \{1, \ldots, r\}$, the edge space $e_i$ is a cylinder $C_i$, two ends of $C_i$ are glued to the loops corresponding to $u_i$ in $\bigvee\limits_{i=1}^{n}S^{1}_{(1)}$ and $\bigvee\limits_{i=1}^{n}S^{1}_{(2)}$. The fundamental group of $X(U)$ is called a \textit{(Baumslag) double of $F_n$ along $U$}, and is often denoted by $D(U)$; and $X(U)$ is called the \emph{natural presentation complex for $D(U)$}. Notice that if $\varphi: F_n\to F_n$ is an automorphism and $U'=\varphi(U)$, then $D(U)\cong D(U')$; and if one replaces some words in $U$ by their conjugates to get a new list $U'$, one still has $D(U)\cong D(U')$. Therefore, it is possible to assume that the list $U$ is \textit{minimal}, consisting of \textit{cyclically reduced words}.\smallbreak
  Given a word $w = x_1\cdots x_m\in F_n$, where $x_1, \ldots, x_m\in \mathcal{A}_n\cup \mathcal{A}^{-1}_n$, the \textit{length-two cyclic subwords} of $w$ are the words in the list $[x_1 x_2, x_2 x_3, \ldots, x_{m-1} x_m, x_m x_1]$. Suppose that we have a list $U$ of cyclically reduced words in $F_n$, the \textit{Whitehead graph} $W(U)$ associated to $U$ is constructed as follows
  \begin{itemize}
      \item[(i)] The vertex set of $W(U)$ is $\mathcal{A}_n\cup \mathcal{A}^{-1}_n$.
      \item[(ii)] Each length-two subword $x\cdot y$ of a word in $U$ corresponds to an edge of $W(U)$ joining $x$ and $y^{-1}$.  
  \end{itemize}
  \begin{note}\upshape
      In general, if $n$ is a natural number, a \textit{Whitehead graph on $2n$ vertices} is the Whitehead graph $W(U)$ of some list of words $U\subseteq F_n$.
  \end{note}
  We give an additional structure on $W(U)$, the \emph{edge-pairing}. Given a vertex $a$ of $W(U)$, denote the set of all edges adjacent to $a$ in $W(U)$ by $\delta_U(a)$, or simply $\delta(a)$ if there is no ambiguity. The \emph{edge-pairing} on $W(U)$ is a collection of maps $(\phi_{a, U})_{a \in \mathcal{A}_n\cup \mathcal{A}^{-1}_n}$, where each map is a bijection from $\delta(a)$ to $\delta(a^{-1})$ defined as follows. Suppose that $w = x_1 \cdots x_m\in U$, denote $x_{m+1} = x_1, x_0 = x_m$, and assume that $x_{i} = a$ for some $i\in \{1, \ldots, m\}$. The length-two subword $x_i x_{i+1}$ corresponds to an edge $e$ in $\delta(a)$, and $\phi_{a, U}(e)$ is defined to be the edge corresponding to the length-two subword $x_{i-1} x_i$ of $w$. If the list $U$ is clear from context, and if $e$ is an edge of $W(U)$, and $e$ is adjacent to $a\in \mathcal{A}_n \cup \mathcal{A}^{-1}_n$, denote $\phi_{a, U}(e) = e^{a^{-1}}$. \smallbreak
  We introduce the notion of \textit{polygonality}, which provides a sufficient condition for a double to contain a hyperbolic surface subgroup, see \cref{polygonality}. Given a list of words in $F_n$ denoted by $U$, let $Z(U)$ be the presentation 2-complex of $F_n/\langle\langle U \rangle \rangle$. By definition, the space $Z(U)$ is a two-dimensional CW-complex with $\operatorname{Cay}(F_n)/F_n = \bigvee\limits_{i=1}^{n}S^1$ as its 1-skeleton; for each word $w = x_1 \cdots x_l$ in $U$, attach the boundary of a 2-disk $D_w$ along the loop reading the word $w$, in other words, the boundary $\partial D_w$ is regarded as an $l$-gon and is glued to $\operatorname{Cay}(F_n)/F_n$ along the loop $x_1 \cdots x_l$. The 0-skeleton of $Z(U)$ consists of a single vertex $v$; the link of this vertex is the Whitehead graph for $U$ by identifying the incoming (outgoing, respectively) portion of a loop $\alpha_i \in \mathcal{A}_n$ with the vertex $\alpha_i$ ($\alpha^{-1}_i$, respectively) in $W(U)$.\smallbreak
  Let $P_1, \ldots, P_m$ be a set of topological 2-disks, each of the disks is equipped with a graph structure on its boundary, such a disk is called a \textit{polygonal disk}. A \textit{side-pairing} on $P_1, \ldots, P_m$ is an equivalence relation on the edges of $\partial P_1, \ldots, \partial P_m$, such that each equivalence class consists of two edges, along with a homeomorphism between the two sides of each equivalence class. Given a side-pairing $\sim$ on $P_1, \ldots, P_m$, we get a closed surface $S = \bigsqcup\limits_{i=1}^{m}P_i / \sim$ with a natural two-dimensional CW-complex structure.  
  \defi[\cite{kim2010geometricity, Kim_2010_wilton}]\upshape
  Let $F_n$ be the free group of rank $n$ with a given free basis $\mathcal{A}_n$. A list $U$ of cyclically reduced words in $F_n$ is called \textit{polygonal} with respect to $\mathcal{A}_n$ if
  \begin{enumerate}
      \item There exists a side-pairing $\sim$ on some polygonal disks $P_1, \ldots, P_m$, which gives a CW-complex $S = \bigsqcup\limits_{i=1}^{m}P_i / \sim$. Denote by $S^{(1)}$ the 1-skeleton of $S$.
      \item There exists a locally injective graph homomorphism $S^{(1)}\to \operatorname{Cay}(F_n)/F_n$, such that \begin{enumerate}
          \item For each $i$, the composition $\partial P_i \hookrightarrow S^{(1)} \to \operatorname{Cay}(F_n)$ satisfies that the cycle $\partial P_i$ is mapped to a loop corresponding to a nontrivial power of a word in $U$.
          \item The Euler characteristic of $S$ is less than $m$.
      \end{enumerate}
      In such case $S$ is called a \textit{$U$-polygonal surface}.
  \end{enumerate}
  \ed
  The main implication of polygonality is the following.
  \thm[\cite{kim2010geometricity, Kim_2010_wilton}]\label{polygonality}
  If $U$ is a polygonal list of words in $F_n$, then a finite cover of $X(U)$ contains a $\pi_1$-injective, closed hyperbolic surface.
  \et
  \subsection{Tiling Conjecture and its combinatorial formulation}
  A list $U$ of words in $F_n$ is said to be \textit{diskbusting} if there does not exist nontrivial subgroups $A, B$ of $F_n$ such that $F_n\cong A * B$, and each word in $U$ is conjugated into $A$ or $B$. It can be shown that the list $U$ is diskbusting if and only if $D(U)$ is one-ended \cite{Gordon_2010}. Tiling Conjecture \cite{kim2010geometricity, Kim_2010_wilton} states that a minimal and diskbusting list of cyclically reduced words in $F_n$ is polygonal, provided that $n>1$.\smallbreak
  Tiling Conjecture can be reformulated into a graph-theoretic statement, as given in \cite{kim2010}; this reformulation is our main consideration in this article.\smallbreak
  \defi\upshape
  Given a list $U$ of cyclically reduced words in $F_n$, the Whitehead graph $W(U)$ is said to be \textit{pairwise well-connected} if for all $a\in \mathcal{A}_n$, there exists $\operatorname{deg}(a) = \operatorname{deg}(a^{-1})$ edge-disjoint paths from vertex $a$ to vertex $a^{-1}$ in $W(U)$.
  \ed
  \begin{proposition}[\cite{Stallings+1999+317+330, stong, heegaard, whitehead}]\label{minimal_diskbusting_and_pairwise}
  Given a list $U$ of cyclically reduced words in $F_n$, the list $U$ is minimal and diskbusting if and only if the Whitehead graph $W(U)$ is connected and pairwise well-connected.      
  \end{proposition}
  \defi[\cite{kim2010}]\upshape
  Given a list $U$ of cyclically reduced words in $F_n$, a non-empty finite list $\mathcal{C}$ of cycles of $W(U)$ is called \textit{balanced} if\begin{itemize}
      \item[(i)] $\mathcal{C}$ has at least one cycle of length at least three.
      \item[(ii)] For each pair of edges $e$ and $f$ incident with a vertex $v^{-1}$, the number of cycles in $\mathcal{C}$ containing both $e$ and $f$ is equal to the number of cycles in $\mathcal{C}$ containing both $e^v$ and $f^v$. 
  \end{itemize} 
  \ed
  \begin{proposition}[{\cite[Lemma 10]{kim2010}}]\label{lemma_10}
      Let $n>1$. A list $U$ of cyclically reduced words in $F_n$ is polygonal if and only if $W(U)$ admits a balanced list of cycles.
  \end{proposition}
  In light of \cref{minimal_diskbusting_and_pairwise}, one can convert the properties of a minimal, diskbusting list of words into graph-theoretic properties on the corresponding Whitehead graph. Furthermore, by \cref{lemma_10}, the notion of polygonality is also converted into a combinatorial property, which together shows that Tiling Conjecture is equivalent to the following conjecture.
   \begin{conjecture}[{\cite[Conjecture 11]{kim2010}}]\label{conj11}
       A connected and pairwise well-connected Whitehead graph on at least four vertices admits a balanced list of cycles.
   \end{conjecture}
   \section{Proof of \cref{main_thm}}\label{main_proof_section}
   We introduce the notion of \emph{uniform} lists of cycles of a Whitehead graph.
   \defi\upshape
   Given a list $U$ of cyclically reduced words in $F_n$, a non-empty finite list $\mathcal{C}$ of cycles of $W(U)$ is called \textit{uniform} if every edge appears in the same number of cycles in $\mathcal{C}$.
   \ed
   In \cite{kim2010}, Kim and Oum proved that a connected, pairwise well-connected Whitehead graph $W$ admits a uniform balanced list of cycles if $W$ has four vertices, or if $W$ is a regular graph. This motivates us to make the following conjecture, which is obviously stronger than \cref{conj11}.
\begin{conjecture}\label{stronger}
       A connected and pairwise well-connected Whitehead graph on at least four vertices admits a uniform and balanced list of cycles.
   \end{conjecture}
   \begin{note}\upshape\label{imp_remark}
   Let $W$ be a Whitehead graph having $2n$ vertices. Then \cref{stronger} is true if $n=2$ or if the graph $W$ is a regular graph, as stated in {\cite[Theorem 12, Theorem 24]{kim2010}}. We use induction to prove a special case of \cref{stronger}, which implies \cref{main_thm}. 
   \end{note}
   \subsection{Inductive arguments}
   For any graph $G$, a graph $H$ is called a \textit{subdivision} of $G$ if $H$ is obtained from $G$ by replacing each edge by a path of length at least one. Firstly, we have the following proposition.
   \begin{proposition}\label{34}
   Assume that $G$ is a Whitehead graph and $H$ is another Whitehead graph which is a subdivision of $G$ (obtained by associating new letters and their inverses to new vertices). Then
   \begin{itemize}
      \item[i)] If $H$ is connected and pairwise well-connected, then so is $G$. 
   \item[ii)] If \cref{stronger} holds for $G$, then \cref{stronger} also holds for $H$.
   \end{itemize}
   \end{proposition}
   \begin{proof}
   The first part is quite obvious, while the second statement has been implicitly written in \cite{kim2010}. For completeness, we provide a detailed proof of the second part. Let $\mathcal{C}$ be a uniform and balanced list of cycles of $G$. Replace each edge of a cycle in $\mathcal{C}$ with the corresponding path in $H$, we obtain a list $\mathcal{C}'$ of cycles in $H$. For each pair $e, f$ of edges in $H$ that are incident with a vertex $v$ not in $G$, because the path $ef$ corresponds to an edge of $G$, therefore, the number of cycles in $\mathcal{C}'$ containing both $e$ and $f$ equals the number of appearances of the corresponding edge of $G$ in $\mathcal{C}$. The same situation holds for $e^v, f^v$. As each edge in $G$ appears the same number of times in $\mathcal{C}$ (since $\mathcal{C}$ is a uniform and balanced list of cycles), it follows that the number of cycles in $\mathcal{C}'$ containing both $e$ and $f$ is equal to the number of cycles in $\mathcal{C}'$ containing both $e^v$ and $f^v$. Meanwhile, for each pair $e, f$ of edges in $H$ that are incident with a vertex $v$ in $G$, the number of cycles in $\mathcal{C}'$ containing both of them equals the number of cycles in $\mathcal{C}$ containing both of the corresponding two edges in $G$. Similarly, the same situation holds for $e^v, f^v$. Therefore, the list $\mathcal{C}'$ is balanced and it is also obviously uniform. 
   \end{proof}
   \bl\label{inductive_process}
Let $k\ge 2$ be a positive integer. If \cref{stronger} is true for all Whitehead graphs having $2m$ vertices and at most $2m+k$ edges, for all $m\in \{2, \ldots, k-1\}$, then \cref{stronger} is true for all Whitehead graphs having $2n$ vertices and at most $2n+k$ edges, for all $n\in\mathbb{N},$ $n\ge 2$.
\el
\begin{proof}
Let $G$ be a Whitehead graph having $2n$ vertices and at most $2n+k$ edges. We proceed by induction on $n$. Firstly, if $n\in \{2, \ldots, k-1\}$, \cref{stronger} is true by the hypothesis. If for some integer $n\geq k-1$, \cref{stronger} is true, we prove \cref{stronger} for all Whitehead graphs having $2(n+1)$ vertices and at most $2(n+1)+k$ edges. \smallbreak
    Because $n+1\ge k$, either $G$ is $3$-regular or there exists a pair of vertices $v, v^{-1}$ such that $\operatorname{deg}(v) = \operatorname{deg}(v^{-1}) = 2$. If the former case happens, then \cref{stronger} holds for $G$ by {\cite[Theorem 12]{kim2010}}. If the latter case happens, consider the following possibilities.
    \begin{itemize}
\item $v$ and $v^{-1}$ are not adjacent. Let $\delta(v) = \{x, y\}$ and $\delta(v^{-1}) = \{z, t\}$. If $x = y$, then by removing $\operatorname{deg}(x) - 2$ edges connecting $x$ and vertices in $V(G)\setminus \{v\}$, we see that $x$ and $x^{-1}$ are disconnected, violating the condition of pairwise well-connectedness. Therefore $x\neq y$, and similarly, we have $z\neq t$. Delete vertices $v, v^{-1}$ of $V(G)$, and add a new edge $e_0$ between $x$ and $y$, a new edge $e_1$ between $z$ and $t$, we obtain a new graph $G'$. The graph $G'$ is a Whitehead graph having $2n$ vertices and at most $2n+k$ edges. It can also be seen that $G$ is a subdivision of $G'$ in the sense of Proposition \ref{34}. By Proposition \ref{34}, $G'$ is connected and pairwise well-connected, thus \cref{stronger} holds for $G'$ by the induction hypothesis. By Proposition \ref{34} again, \cref{stronger} holds for $G$.  
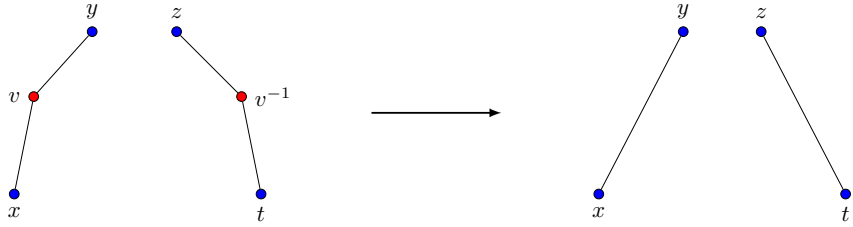
\begin{figure}[H]
    \centering
    \resizebox{0.7 \textwidth}{!}{\begin{tikzpicture}[
    dot/.style={circle, draw, fill=#1, inner sep=1.5pt, minimum size=4pt},
    dot/.default=blue
]

    \node[dot] (x1) at (0,0) [label=below:$x$] {};
    \node[dot=red] (v) at (0.3,1.5) [label=left:$v$] {};
    \node[dot] (y1) at (1.2,2.5) [label=above:$y$] {};
    \draw (x1) -- (v) -- (y1);

    \node[dot] (z1) at (2.5,2.5) [label=above:$z$] {};
    \node[dot=red] (vi) at (3.5,1.5) [label=right:$v^{-1}$] {};
    \node[dot] (t1) at (3.8,0) [label=below:$t$] {};
    \draw (z1) -- (vi) -- (t1);

    \draw[-latex,thick] (5.5,1.25) -- (7.5,1.25);

    \node[dot] (x2) at (9,0) [label=below:$x$] {};
    \node[dot] (y2) at (10.3,2.5) [label=above:$y$] {};
    \draw (x2) -- (y2);

    \node[dot] (z2) at (11.5,2.5) [label=above:$z$] {};
    \node[dot] (t2) at (12.8,0) [label=below:$t$] {};
    \draw (z2) -- (t2);

\end{tikzpicture}}
    \caption{Delete $v$, $v^{-1}$ and replace $(xv, vy)$ by $xy$, replace $(tv^{-1},v^{-1}z)$ by $tz$.}
    \label{fig1}
\end{figure}

\item $v$ and $v^{-1}$ are adjacent. It is not hard to see that $v$ is also adjacent to $x\neq v^{-1}$, and $v^{-1}$ is adjacent to $y\notin \{v, x\}$. Deleting vertices $v, v^{-1}$ of $V(G)$, and add a new edge $e_1$ between $x$ and $y$, we obtain a new graph $G'$. The graph $G'$ is a Whitehead graph having $2n$ vertices and at most $2n+k$ edges. 
\begin{figure}[H]
    \centering
    \begin{tikzpicture}[
    dot/.style={circle, draw, fill=#1, inner sep=1.5pt, minimum size=4pt},
    dot/.default=blue,
    >=latex 
]

    \node[dot] (x1) at (0,0) [label=below:$x$] {};
    \node[dot=red] (v) at (0.2,1.5) [label=left:$v$] {};
    \node[dot=red] (vi) at (2.8,1.4) [label=right:$v^{-1}$] {};
    \node[dot] (y1) at (3,0) [label=below:$y$] {};
    
    \draw (x1) -- (v) -- (vi) -- (y1);

    \draw[->, thick] (3.8,0.7) -- (5.5,0.7);

    \node[dot] (x2) at (6.2,0.4) [label=below:$x$] {};
    \node[dot] (y2) at (8.5,0.6) [label=below:$y$] {};
    
    \draw (x2) -- (y2);

\end{tikzpicture}
    \caption{Delete $v$, $v^{-1}$ and replace $(xv, vv^{-1}, v^{-1}y)$ by $xy$.}
    \label{fig2}
\end{figure}
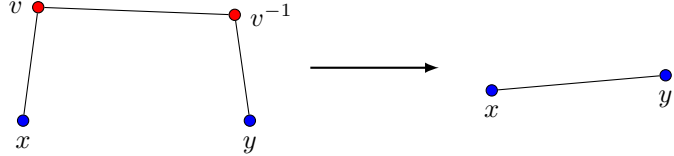
Similarly as above, $G$ is also a subdivision of $G'$ in the sense of Proposition \ref{34}, thus $G'$ is connected and pairwise well-connected. Again, by the induction hypothesis, \cref{stronger} holds for $G'$. By Proposition 3.4, \cref{stronger} holds for $G$.\qedhere
\end{itemize}
\end{proof}
\begin{note}
\upshape 
The fact that the degrees of $v$ and $v^{-1}$ are two is crucial. If $v$ and $v^{-1}$ have larger degrees, then there is a problem with the removal of $v$ and $v^{-1}$: which edges can represent all paths of the form $x \rightarrow v \rightarrow y$? A natural idea is to replace each path $x \rightarrow v \rightarrow y$ with the edge $xy$. However, this idea doesn't turn cycles into cycles. 
\end{note}
\subsection{Proof of \cref{main_thm}}
\defi
\upshape 
Let $U$ be a list of words in $F_n$, the sum of lengths of all words in $U$ is called the \textit{length} of $U$.
\ed
We introduce the notion of \emph{equivalence} of lists of words.    
\defi\upshape\label{defi:equivalence}
Let $U$ and $U'$ be lists of words in $F_n$, the free group of rank $n$ with a free basis $\mathcal{A}_n$. We say that $U$ and $U'$ are \textit{equivalent} if there exists a graph isomorphism $\psi$ from $W(U)$ to $W(U')$, such that $\psi(a) = a$, for all $a\in \mathcal{A}_n\cup\mathcal{A}^{-1}_n$.
\ed
\begin{note}\upshape
    From \cref{defi:equivalence}, it is not hard to see that if $U$ and $U'$ are equivalent lists of cyclically reduced words, then $U$ is minimal and diskbusting if and only if $U'$ is. However, the edge-pairings of $W(U)$ and $W(U')$ may be different. Furthermore, if two lists $U$ and $U'$ have the same list of length-two subwords, then $U$ and $U'$ are equivalent. 
\end{note}
\bl\label{single}
If $U$ is a minimal diskbusting list of words in $F_n$, then $U$ is equivalent to a list having a single word in $F_n$ having the same length as $U$.
\el
\begin{proof}
    Suppose that the list $U$ contains more than one word, because $U$ is minimal and diskbusting, we can assume without loss of generality that there exist $w_1, w_2\in U$ such that two words $w_1$ and $w_2$ both contain a letter $a\in \mathcal{A}_n$. By cyclically permuting $w_1$ and $w_2$, we obtain two new words $u_1$ and $u_2$ both having letter $a$ as their first letters. Observe that the list of length-two subwords of $[w_1, w_2]$ is the same as the list of length-two subwords of $[u_1\cdot u_2]$, and the two lists have the same length. By repeatedly doing the above procedure, one sees that the list $U$ is equivalent to a list containing a single word with the same length as $U$.
\end{proof}
\begin{proposition}\label{6_10_case}
Let $G$ be a connected and pairwise well-connected Whitehead graph having $6$ vertices and $10$ edges. If $G$ does not contain vertices of valence $2$, then there exists a list $\mathcal{C}$ of cycles of $G$, satisfying the following conditions.
\begin{itemize}
    \item[(i)] Each edge of $G$ appears exactly in exactly $6$ cycles in $\mathcal{C}$.
    \item[(ii)] Each vertex has valence either three or four. Moreover, for $d\in \{3, 4\}$, if $v$ has valence $d$, then for each pair of distinct edges $e, f\in \delta(v)$, there are exaclty $(6-d)$ cycles in $\mathcal{C}$ containing both $e$ and $f$. 
\end{itemize}
In particular, \cref{stronger} is true for all Whitehead graphs having $6$ vertices and at most $10$ edges.
\end{proposition}
\begin{proof}
    Let $F_3$ be the free group on three letters $a, b, c$, denote $A = a^{-1}, B = b^{-1}, C = c^{-1}$. Let $G = (V(G), E(G))$ be a connected and pairwise well-connected Whitehead graph having $6$ vertices and $10$ edges, we have $V(G) = \{a, A, b, B, c, C\}$. Without loss of generality, assume that $\operatorname{deg}(a) = \operatorname{deg}(A) = 4$, and $\operatorname{deg}(b) = \operatorname{deg}(B) = \operatorname{deg}(c) = \operatorname{deg}(C) = 3$. First, observe that if $G$ has parallel edges, they must be edges between the pair $(a, A)$, the pair $(b, B)$ or the pair $(c, C)$.\smallbreak 
    Since the statement of \cref{6_10_case} does not depend on edge-pairings, we can identify $G$ with a Whitehead graph of a single word by \cref{single}. That allows us to prove \cref{6_10_case} by considering all possible Whitehead graphs, aided by a computer. In each case, if $G$ is the Whitehead graph under consideration, then $G$ is identified with the Whitehead graph of a list having one word, from the word $w$ in the list, we get a matrix $M_w$ whose rows are incidence vectors of cycles in $G$, and a matrix $N_w$ whose the $(i, j)$ entry equals $1$ if the $i^{th}$ cycle contains the $j^{th}$ pair of adjacent edge, and equals $0$ otherwise. The problem now is to find a nonzero row vector $v_w$ consisting of nonnegative integer entries, such that the list of cycles corresponding to $v_w$ satisfies the conditions stated in \cref{6_10_case}. We give the explicit matrices and vectors for one case, the details for the other cases are given in the ancillary files.    
    \begin{enumerate}
        \item If there is no edge between $a$ and $A$, then $a$ and $A$ must be adjacent to $b, B, c, C$. Thus there are at most one edge between $b, B$, similarly for $c, C$.
        \begin{enumerate}
            \item If there is no edge between $b$ and $B$, and between $c$ and $C$, then we can assume that $bc, BC\in E(G)$. Then $G$ corresponds to the word $w_0 = 'acAbAcabCb'$. Once we have the word $w_0$, from the list of length-two subwords of $w_0$, we have a natural way to label the edges of $G$. In this case, the graph $G$ has no parallel edges, so we can represent each edge by its two endpoints as follows. 
            \begin{equation*}
                bA: 0, aC: 1, ca: 2, AB: 3, ba: 4, AC: 5, cA: 6, aB: 7, bc: 8, CB: 9.
            \end{equation*}We also arrange the set of all pairs of adjacent edges in $G$ by the lexicographical ordering, based on the above labelling of the edges. For example, the first pair is $(0, 3)$, corresponding to the pair of edges $(bA, AB)$ of $G$.
            \begin{eqnarray*}
                0: (0, 3), 1: (0, 4), 2: (0, 5), 3: (0, 6), 4: (0, 8), 5: (1, 2),\\ 6: (1, 4), 7: (1, 5), 8: (1, 7), 9: (1, 9), 10: (2, 4), 11: (2, 6),\\ 12: (2, 7), 13: (2, 8), 14: (3, 5), 15: (3, 6), 16: (3, 7), 17: (3, 9),\\ 18: (4, 7), 19: (4, 8), 20: (5, 6), 21: (5, 9), 22: (6, 8), 23: (7, 9).
            \end{eqnarray*}Having fixed the ordering of the edges and the pairs of adjacent edges, we have the matrices $M_{w_0}$ and $N_{w_0}$ as follows.
            \begin{equation*}
                M_{w_0} = \begin{pmatrix}
                    0 &0& 0& 0& 1& 1& 1& 1& 1& 1\\
                    0 &0 &0 &1 &0 &1 &0& 0& 0& 1\\
 0 &0 &0 &1 &1 &0 &1& 1& 1& 0\\
 0 &0 &1 &0 &0 &1 &1 &1 &0 &1\\
 0 &0 &1 &0 &1 &0 &0 &0 &1 &0\\
 0 &0 &1 &1 &0 &0 &1 &1 &0 &0\\
 0 &1 &0 &0 &0 &0 &0 &1 &0 &1\\
 0 &1 &0 &0 &1 &1 &1 &0 &1 &0\\
 0 &1 &0 &1 &0 &1 &0 &1 &0 &0\\
 0 &1 &0 &1 &1 &0 &1 &0 &1 &1\\
 0 &1 &1 &0 &0 &1 &1 &0 &0 &0\\
 0 &1 &1 &1 &0 &0 &1 &0 &0 &1\\
 1 &0 &0 &0 &0 &0 &1 &0 &1 &0\\
 1 &0 &0 &0 &1 &1 &0 &1 &0 &1\\
 1 &0 &0 &1 &1 &0 &0 &1 &0 &0\\
 1 &0 &1 &0 &0 &1 &0 &1 &1 &1\\
 1 &0 &1 &0 &1 &0 &1 &0 &0 &0\\
 1 &0 &1 &1 &0 &0 &0 &1 &1 &0\\
 1 &1 &0 &0 &1 &1 &0 &0 &0 &0\\
 1 &1 &0 &1 &1 &0 &0 &0 &0 &1\\
 1 &1 &1 &0 &0 &1 &0 &0 &1 &0\\
 1 &1 &1 &1 &0 &0 &0 &0 &1 &1
                \end{pmatrix}.
            \end{equation*}
            \begin{equation*}
                N_{w_0} = \begin{pmatrix}
 0 &0 &0 &0 &0 &0 &0 &0 &0 &0 &0 &0 &0 &0 &0 &0 &0 &0 &1 &1 &1 &1 &1 &1\\
 0 &0 &0 &0 &0 &0 &0 &0 &0 &0 &0 &0 &0 &0 &1 &0 &0 &1 &0 &0 &0 &1 &0 &0\\
 0 &0 &0 &0 &0 &0 &0 &0 &0 &0 &0 &0 &0 &0 &0 &1 &1 &0 &1 &1 &0 &0 &1 &0\\
 0 &0 &0 &0 &0 &0 &0 &0 &0 &0 &0 &1 &1 &0 &0 &0 &0 &0 &0 &0 &1 &1 &0 &1\\
 0 &0 &0 &0 &0 &0 &0 &0 &0 &0 &1 &0 &0 &1 &0 &0 &0 &0 &0 &1 &0 &0 &0 &0\\
 0 &0 &0 &0 &0 &0 &0 &0 &0 &0 &0 &1 &1 &0 &0 &1 &1 &0 &0 &0 &0 &0 &0 &0\\
 0 &0 &0 &0 &0 &0 &0 &0 &1 &1 &0 &0 &0 &0 &0 &0 &0 &0 &0 &0 &0 &0 &0 &1\\
 0 &0 &0 &0 &0 &0 &1 &1 &0 &0 &0 &0 &0 &0 &0 &0 &0 &0 &0 &1 &1 &0 &1 &0\\
 0 &0 &0 &0 &0 &0 &0 &1 &1 &0 &0 &0 &0 &0 &1 &0 &1 &0 &0 &0 &0 &0 &0 &0\\
 0 &0 &0 &0 &0 &0 &1 &0 &0 &1 &0 &0 &0 &0 &0 &1 &0 &1 &0 &1 &0 &0 &1 &0\\
 0 &0 &0 &0 &0 &1 &0 &1 &0 &0 &0 &1 &0 &0 &0 &0 &0 &0 &0 &0 &1 &0 &0 &0\\
 0 &0 &0 &0 &0 &1 &0 &0 &0 &1 &0 &1 &0 &0 &0 &1 &0 &1 &0 &0 &0 &0 &0 &0\\
 0 &0 &0 &1 &1 &0 &0 &0 &0 &0 &0 &0 &0 &0 &0 &0 &0 &0 &0 &0 &0 &0 &1 &0\\
 0 &1 &1 &0 &0 &0 &0 &0 &0 &0 &0 &0 &0 &0 &0 &0 &0 &0 &1 &0 &0 &1 &0 &1\\
 1 &1 &0 &0 &0 &0 &0 &0 &0 &0 &0 &0 &0 &0 &0 &0 &1 &0 &1 &0 &0 &0 &0 &0\\
 0 &0 &1 &0 &1 &0 &0 &0 &0 &0 &0 &0 &1 &1 &0 &0 &0 &0 &0 &0 &0 &1 &0 &1\\
 0 &1 &0 &1 &0 &0 &0 &0 &0 &0 &1 &1 &0 &0 &0 &0 &0 &0 &0 &0 &0 &0 &0 &0\\
 1 &0 &0 &0 &1 &0 &0 &0 &0 &0 &0 &0 &1 &1 &0 &0 &1 &0 &0 &0 &0 &0 &0 &0\\
 0 &1 &1 &0 &0 &0 &1 &1 &0 &0 &0 &0 &0 &0 &0 &0 &0 &0 &0 &0 &0 &0 &0 &0\\
 1 &1 &0 &0 &0 &0 &1 &0 &0 &1 &0 &0 &0 &0 &0 &0 &0 &1 &0 &0 &0 &0 &0 &0\\
 0 &0 &1 &0 &1 &1 &0 &1 &0 &0 &0 &0 &0 &1 &0 &0 &0 &0 &0 &0 &0 &0 &0 &0\\
 1 &0 &0 &0 &1 &1 &0 &0 &0 &1 &0 &0 &0& 1 &0 &0 &0 &1& 0 &0 &0 &0 &0 &0
                \end{pmatrix}.
            \end{equation*}
            By definition, the first row of $M_{w_0}$ and the first row of $N_{w_0}$ both correspond to the cycle having the edges $4, 5, 6, 7, 8, 9$, and we have analogous correspondences for the remaining rows of the matrices. Let \begin{equation}
                v_{w_0} = \begin{pmatrix}
                1& 0& 0& 1& 0& 0& 0& 0& 2& 2& 0& 0& 0& 1& 0& 0& 2& 1& 0& 0& 1& 1
            \end{pmatrix},\tag{3.1}\label{eq:case1a}
            \end{equation}we see that \begin{equation*}
                v_{w_0} \cdot M_{w_0} = \begin{pmatrix}
                    6 & 6 & 6 & 6 & 6 & 6 & 6 & 6 & 6 & 6
                \end{pmatrix},
            \end{equation*}
            \begin{equation*}
                v_{w_0} \cdot N_{w_0} = \begin{pmatrix}
                    2& 3 &2& 2 &3 &2& 2& 3& 2& 3& 2& 3& 2& 3& 2& 2& 3& 3& 2& 3& 2& 3& 3& 3
                \end{pmatrix},
            \end{equation*}and it is easy to check that the list of cycles corresponding to $v_{w_0}$ has at least one cycle of length at least three. This indicates that the list of cycles corresponding to $v_{w_0}$ satisfies the desired properties. 
            \item If $b$ and $B$ are adjacent, then $c$ and $C$ are adjacent. The graph $G$ corresponds to the word $w_1 = 'baBaCaccab'$.
        \end{enumerate}
        \item If there is one edge between $a$ and $A$, then assume that $a$ is also adjacent to $b$, $c$, and $C$. Then we have the following cases.\begin{enumerate}
            \item If $A$ is adjacent to $b$, $c$, and $C$, then $B$ must be adjacent to $b$, $c$, and $C$. Then, the graph $G$ corresponds to $w_2 = 'aacaCaBBcb'$.
            \item If $A$ is adjacent to $c, C$ and $B$, then either there is one edge between $b$ and $B$, this case corresponds to the word $w_3 = 'CaBaacaCbb'$,
            or there are two edges between $b$ and $B$, and one edge between $c$ and $C$, this case corresponds to the word $w_4 = 'BBBaCaccaa'$.
            \item If $A$ is adjacent to $b, C, B$, then $c$ must be adjacent to $B$. If $c$ is adjacent to $b$, we get the corresponding word $w_5 = 'aaBaCacBcb'$.
            If $c$ is not adjacent to $b$, then the vertices $b$ and $B$ are adjacent, also the vertices $c$ and $C$ are adjacent. We get the corresponding word being $w_6 = 'CCacbbaaBa'$.
        \end{enumerate}
        \item If there are two edges between $a$ and $A$, then we have the following cases.\begin{enumerate}
            \item If $a$ is adjacent to $b$ and $c$, and $A$ is adjacent to $B$ and $C$, and furthermore, the vertices $b$ and $B$ are not adjacent, then $b$ is adjacent to $c$ and $C$, and $B$ is adjacent to $c$ and $C$. In this case, the graph $G$ corresponds to the word $w_7 = 'CaBaaaCbcb'$.
            If there are two edges connecting $b$ and $B$, then there are two edges connecting $c$ and $C$, thus $G$ corresponds to $w_8 = 'CCCaaaBBBa'$.
If there is one edge connecting $b$ and $B$, then there is one edge between $c$ and $C$, we can assume that $b$ is adjacent to $c$ and $B$ is adjacent to $C$. In this case, the graph $G$ corresponds to $w_9 = 'CCaaaBBcBa'$.
            \item Assume that $a$ is adjacent to $b$ and $c$, and $A$ is adjacent to $c$ and $B$. Suppose $b$ is not adjacent to $C$, then because $\operatorname{deg}(C) = 3$, there must be at least 2 edges between $c$ and $C$, violating the condition that $\operatorname{deg}(c) = 3$. Thus $b$ is adjacent to $C$. Furthermore, we have $B$ is adjacent to $C$, and there is one edge connecting $b$ and $B$, one edge connecting $c$ and $C$. Hence, the graph $G$ corresponds to the word $w_{10} = 'ccaBaaaCbb'$.
            \item If $a$ is adjacent to $b$ and $c$, and if $A$ is adjacent to $c$ and $b$, then $B$ must be adjacent to $b$, $c$, and $C$, similarly, the vertex $C$ must be adjacent to $b$, $c$, and $C$, violating the condition that $\operatorname{deg}(c) = 3$.
            \item Assume that $a$ is adjacent to $b$ and $B$, and $A$ is adjacent to $c$ and $C$. Suppose that $b$ is not adjacent to $B$, then $b$ must be adjacent to both $C$ and $c$, similarly, the vertex $B$ must be adjacent to $c$ and $C$. Thus $G$ corresponds to the word $w_{11} = 'bAcaaabcbC'$.
If $b$ is adjacent to $B$, then $c$ is adjacent to $C$. Because we can interchange $c$ and $C$ in the above assumptions, without loss of generality, we can also assume that the vertices $b$ and $c$ are adjacent, the vertices $B$ and $C$ are adjacent. In this case, the graph $G$ corresponds to the word $w_{12} = 'bbCbAccaaa'$.
If there are two edges connecting $b$ and $B$, then there are two edges connecting $c$ and $C$, then by removing two edges between $a$ and $A$, we disconnect $a$ and $A$, which violates the condition of well-connectedness.
            \item If $a$ is adjacent to $b$ and $B$, and $A$ is adjacent to $b$ and $B$, then we can assume that $b$ is adjacent to $c$, and $B$ is adjacent to $C$. Then there are two edges connecting $c$ and $C$, thus $G$ corresponds to $w_{13}='BAAABaBccc'$.
            \item Assume that $a$ is adjacent to $b$ and $B$, and $A$ is adjacent to $b$ and $c$. Suppose that $b$ and $C$ are adjacent, then $C$ and $B$ are adjacent, and there is one edge connecting $c$ and $C$. The corresponding word is $w_{14}='ccbaaabACb'$.
If $b$ and $C$ are not adjacent, then $C$ must be adjacent to $B$, and there are two edges between $c$ and $C$. This case corresponds to the word $w_{15}='CCCbbaaabA'$.
        \end{enumerate}
        \item If $a$ and $A$ are connected by three edges, and $a$ is adjacent to $b$, consider the following cases.
        \begin{enumerate}
            \item If $A$ is adjacent to $B$, and $b$ is not adjacent to $B$, then $b$ and $B$ must be adjacent to both $c$ and $C$. The corresponding word is $w_{16}='bccbCbAAAA'$.
If $b$ is adjacent to $B$, then we can assume that $b$ is adjacent to $c$ as well. Then $B$ is adjacent to $C$, and there are two edges between $c$ and $C$. The corresponding word is $w_{17}='BBcccBaaaa'$.
            \item If $A$ is adjacent to $c$, then $B$ must be adjacent to $b$. If there is one edge between $b$ and $B$, then $B$ is adjacent to $c$ and $C$, and $C$ is adjacent to $b$ and $c$. The corresponding word is $w_{18}='ccbbcaaaaB'$.
Next, if there are two edges between $b$ and $B$, then $C$ is adjacent to $B$, and there are two edges between $c$ and $C$. The corresponding word is $w_{19}='BBBcccaaaa'$.
            \item Finally, if $A$ is adjacent to $b$, then $B$ must be adjacent to $b$, $c$, and $C$, and there are two edges between $c$ and $C$, but we can disconnect $b$ and $B$ by removing the edge between $b$ and $B$, violating the condition of pairwise well-connectedness.  
            \qedhere
        \end{enumerate}
    \end{enumerate}
\end{proof}
The following theorem in \cite{kim2010} serves as an important part in the base step for the inductive process of proving \cref{main_thm}.
\begin{theorem}[{\cite[Theorem 24]{kim2010}}]\label{thm24:kim2010}
    Let $G$ be the Whitehead graph of a minimal diskbusting list of words in $F_2$. Then $G$ admits a balanced list $\mathcal{C}$ of cycles such that each edge of $G$ appears in the same number of cycles in $\mathcal{C}$.
\end{theorem}
Now \cref{main_thm} is an immediate consequence of the following.
\cor\label{thm2.1.1}
If $G$ is a connected and pairwise well-connected Whitehead graph with $2n$ vertices and at most $2n + 4$ edges, then $G$ admits a uniform balanced list of cycles.
\ec
\begin{proof}
Apply \cref{inductive_process} for $k=3$, we have that \cref{stronger} is true for all Whitehead graphs having $2n$ vertices and at most $2n+3$ edges if \cref{stronger} is true for all Whitehead graphs having $4$ vertices and at most $7$ edges. By \cref{thm24:kim2010}, \cref{stronger} is true for all Whitehead graphs having $4$ vertices. By \cref{inductive_process} again, we see that proving \cref{6_10_case} is enough to ensure that \cref{stronger} is true for all Whitehead graphs having $2n$ vertices and at most $2n + 4$ edges. 
\end{proof}
\section{Further findings for the six-vertex case}\label{discussion}
\subsection{Tiling Conjecture for simple graphs on six vertices}
We can repeat the idea of proving \cref{6_10_case} to verify \cref{stronger} for simple Whitehead graphs on six vertices. In detail, we prove that simple Whitehead graphs on six vertices admit lists of cycles that are uniform and balanced with respect to any possible edge pairing.
\begin{proposition}\label{stronger_for_simple_graphs}
    Suppose that $G = (V, E)$ is a simple Whitehead graphs having $6$ vertices, with pair of vertices $(a_1, A_1), (a_2, A_2), (a_3, A_3)$. Then there exists a list $\mathcal{C}$ of cycles of $G$, and $c_0, c_1, c_2, c_3\in \mathbb{N}$ such that\begin{itemize}
        \item[(i)] Each edge of $G$ appears in exactly $c_0$ cycles in $\mathcal{C}$.
        \item[(ii)] For all $i\in \{1, 2, 3\}$, if $e$ and $f$ are two distinct edges such that $\{e, f\}\subseteq \delta(a_i)$ or $\{e, f\}\subseteq \delta(A_i)$, there are exactly $c_i$ cycles in $\mathcal{C}$ containing both $e$ and $f$. 
    \end{itemize}
\end{proposition}
\begin{proof}
    A simple graph $G = (V,E)$ on six vertices has at most $\binom{6}{2} = 15$ edges. If it is the complete graph, we can choose the list $\mathcal{C}$ consisting of all triangles. If $G$ has no more than ten edges, it follows from \cref{6_10_case} and \cref{inductive_process} that the claim holds. Therefore, we only need to consider the case $11 \le |E| \le 14$. From now on, we will rename the vertices $a_1, A_1, a_2, A_2, a_3, A_3$ by $a, A, b, B, c, C$, respectively. As in the proof of \cref{6_10_case}, because the statement of \cref{stronger_for_simple_graphs} does not depend on edge-pairings, we identify each of the following Whitehead graphs with a Whitehead graph of a single word, and then find a list of cycles satisfying the desired conditions. The list of cycles corresponding to each word is explicitly given in the ancillary files.\smallbreak
    If $|E| = 14$, assume without loss of generality that $\operatorname{deg}(a) = \operatorname{deg}(A) = 4$, while the other vertices are of valence $5$, thus $a$ and $A$ are not adjacent, and $G$ corresponds to the Whitehead graph of $w_{20}='caCaBabacbCbbc'$. 
    If $|E| = 13$, there are four vertices of valence four, namely $a, b, A, B$, which yields two cases.
\begin{itemize}
\item[(i)] $a$ is adjacent to $A$, and $b$ is adjacent to $B$. Then $G$ corresponds to the Whitehead graph of $w_{21}='ccaCaabacbCbb'$.
\item[(ii)] $a$ is adjacent to $b$, and $A$ is adjacent to $B$. Then $G$ corresponds to the Whitehead graph of $w_{22}='caCaBabaccbCb'$. 
\end{itemize}If $|E| = 12$, either the graph is $4$-regular or the graph has exactly two vertices of valence three, two of valence four and two of valence five. We only need to check the second case, where the Whitehead graph is shown in the diagram below. In the following figures, vertices of the same color represent pairs of generators that are inverses of each other. 
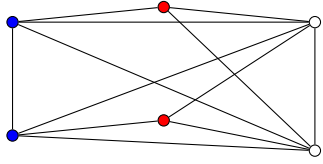
\begin{figure}[H]
    \centering
    \begin{tikzpicture}[
    dot/.style={circle, draw, fill=#1, inner sep=1.5pt, minimum size=4pt},
    dot/.default=blue,
    >=latex 
]

    \node[dot=blue] (a1) at (0,0)  {};
    \node[dot=blue] (a2) at (0,1.5) {};
    \node[dot=red] (b1) at (2,0.2)  {};
    \node[dot=red] (b2) at (2,1.7) {};
    \node[dot=white] (c1) at (4, 1.5)  {};
    \node[dot=white] (c2) at (4,-0.2) {};
    
    \draw (a1) -- (a2);
    \draw (a1) -- (b1);
    \draw (a2) -- (b2);
    \draw (a1) -- (c1);
    \draw (a2) -- (c2);
    \draw (b1) -- (c2);
    \draw (c1) -- (c2);
    \draw (b2) -- (c1);
    \draw (a2) -- (c1);
    \draw (a1) -- (c2);
    \draw (b2) -- (c2);
    \draw (b1) -- (c1);

\end{tikzpicture}
    \caption{The non-regular simple Whitehead graph with $6$ vertices and $12$ edges.}
    \label{fig3}
\end{figure}
In this case, the graph $G$ corresponds to the Whitehead graph of $w_{23}='ccaCaBaacbCb'$.
Lastly, if $|E| = 11$, the valences of $a, b, c$ are $3, 3, 5$ or $3, 4, 4$, in any order. If these are $3, 3, 5$, there are two possible cases for the complement graph of $G$. 
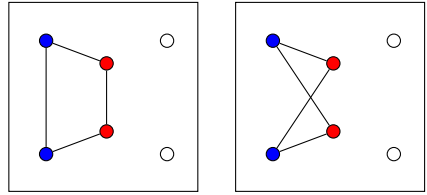
\begin{figure}[H]
    \centering
    \begin{tikzpicture}[
    box/.style={draw, minimum width=2.5cm, minimum height=2.5cm, anchor=south west},
    dot/.style={circle, draw, fill=#1, inner sep=1.2pt, minimum size=5pt},
    dot/.default=blue
]

    \node[box] (b1) at (0,0) {};
    \node[dot=blue] (a1) at (0.5, 2.0) {}; \node[dot=blue] (a2) at (0.5, 0.5) {};
    \node[dot=red]  (a3) at (1.3, 1.7) {}; \node[dot=red]  (a4) at (1.3, 0.8) {};
    \node[dot=white] (a5) at (2.1, 2.0) {}; \node[dot=white] (a6) at (2.1, 0.5) {};
    \draw (a1) -- (a2); 
    \draw (a2) -- (a4); 
    \draw (a4) -- (a3); 
    \draw (a3) -- (a1); 

    \begin{scope}[xshift=3cm]
        \node[box] (b2) at (0,0) {};
        \node[dot=blue] (c1) at (0.5, 2.0) {}; \node[dot=blue] (c2) at (0.5, 0.5) {};
        \node[dot=red]  (c3) at (1.3, 1.7) {}; \node[dot=red]  (c4) at (1.3, 0.8) {};
        \node[dot=white] (c5) at (2.1, 2.0) {}; \node[dot=white] (c6) at (2.1, 0.5) {};
        \draw (c1) -- (c3); 
        \draw (c3) -- (c2); 
        \draw (c2) -- (c4); 
        \draw (c4) -- (c1);
    \end{scope}
\end{tikzpicture}
    \caption{Two cases for the complement graph of $G$, if the valences of $a, b, c$ are $3, 3, 5$.}
    \label{fig4}
\end{figure}
The case on the left corresponds to the word $w_{24}='ccbCbcaCaba'$.
The case on the right corresponds to the word $w_{25}='ccaCaacbCbb'$.
If the valences of $a, b, c$ are $3, 4, 4$, there are four possible cases.
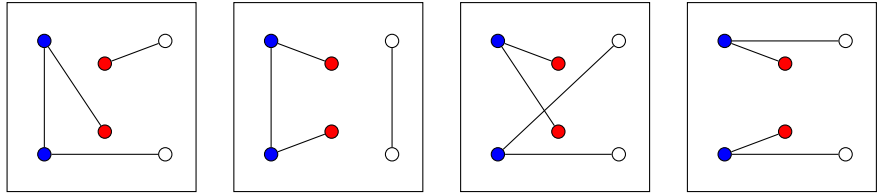
\begin{figure}[H]
    \centering
    \begin{tikzpicture}[
    box/.style={draw, minimum width=2.5cm, minimum height=2.5cm, anchor=south west},
    dot/.style={circle, draw, fill=#1, inner sep=1.2pt, minimum size=5pt},
    dot/.default=blue
]

    \node[box] (b1) at (0,0) {};
    \node[dot=blue] (a1) at (0.5, 2.0) {}; \node[dot=blue] (a2) at (0.5, 0.5) {};
    \node[dot=red]  (a3) at (1.3, 1.7) {}; \node[dot=red]  (a4) at (1.3, 0.8) {};
    \node[dot=white] (a5) at (2.1, 2.0) {}; \node[dot=white] (a6) at (2.1, 0.5) {};
    \draw (a1) -- (a2); 
    \draw (a2) -- (a6); 
    \draw (a1) -- (a4); 
    \draw (a3) -- (a5); 

    \begin{scope}[xshift=3cm]
        \node[box] (b2) at (0,0) {};
        \node[dot=blue] (c1) at (0.5, 2.0) {}; \node[dot=blue] (c2) at (0.5, 0.5) {};
        \node[dot=red]  (c3) at (1.3, 1.7) {}; \node[dot=red]  (c4) at (1.3, 0.8) {};
        \node[dot=white] (c5) at (2.1, 2.0) {}; \node[dot=white] (c6) at (2.1, 0.5) {};
        \draw (c1) -- (c2) (c1) -- (c3) (c2) -- (c4) (c5) -- (c6);
    \end{scope}

    \begin{scope}[xshift=6cm]
        \node[box] (b3) at (0,0) {};
        \node[dot=blue] (d1) at (0.5, 2.0) {}; \node[dot=blue] (d2) at (0.5, 0.5) {};
        \node[dot=red]  (d3) at (1.3, 1.7) {}; \node[dot=red]  (d4) at (1.3, 0.8) {};
        \node[dot=white] (d5) at (2.1, 2.0) {}; \node[dot=white] (d6) at (2.1, 0.5) {};
        \draw (d4) -- (d1); 
        \draw (d1) -- (d3); 
        \draw (d5) -- (d2); 
        \draw (d2) -- (d6); 
    \end{scope}
      \begin{scope}[xshift=9cm]
        \node[box] (b3) at (0,0) {};
        \node[dot=blue] (d1) at (0.5, 2.0) {}; \node[dot=blue] (d2) at (0.5, 0.5) {};
        \node[dot=red]  (d3) at (1.3, 1.7) {}; \node[dot=red]  (d4) at (1.3, 0.8) {};
        \node[dot=white] (d5) at (2.1, 2.0) {}; \node[dot=white] (d6) at (2.1, 0.5) {};
        \draw (d3) -- (d1) -- (d5); 
        \draw (d4) -- (d2) -- (d6);
    \end{scope}
\end{tikzpicture}
    \caption{Four cases for the complement graph of $G$, if the valences of $a, b, c$ are $3, 4, 4$.}
    \label{fig5}
\end{figure}
From left to right, the graph $G$ corresponds to the words $w_{26}= {'cbbcaBaccAB'}$, $w_{27}= {'bbcbCbacaCa'}$, $w_{28}=  {'ccbCbbcABaa'}$, and $w_{29}= {'ccaabacbCbb'}$, respectively. 
\end{proof}
The following corollary is immediate from \cref{stronger_for_simple_graphs}.
\cor\label{simple_graphs}
    Let $G$ be a Whitehead graph having $6$ vertices, assume that $G$ has no parallel edges. Then $G$ admits a uniform balanced list of cycles. 
\ec
\begin{note}\upshape
     \cref{stronger_for_simple_graphs} does not hold for an arbitrary connected, pairwise well-connected Whitehead graph. We provide a counterexample of a connected, pairwise well-connected Whitehead graph $G$ with parallel edges. $G$ has four vertices labeled $a, A, b, B$, and the edge multiset of $G$ is
    \begin{equation*}
        \{(a, A), (a, b) = e_0, (a, b) = e_1, (a, B) = f_0, (a, B)= f_1,(A, b) = k_0, (A, b) = k_1,(A, B) = l_0, (A, B) = l_1\}.  
    \end{equation*}
    \begin{figure}[H]
        \centering
    \begin{tikzpicture}[
    dot/.style={circle, draw, fill=#1, inner sep=1.5pt, minimum size=4pt},
    dot/.default=red,
    >=latex 
]
\node[] at (0, 2.5) {$e_0$};
\draw [black] plot [smooth] coordinates {(0, 0) (0.3, 2.5) (1.5, 4)};
\draw [black] plot [smooth] coordinates {(0, 0) (1.5, 4)};
\node[] at (1.3, 2.5) {$e_1$};
\node[] at (3, 2.5) {$f_1$};
\node[] at (3.3, 3.2) {$f_0$};
\node[] at (1.3, 1) {$k_0$};
\node[] at (1.5, 0.3) {$k_1$};
\node[] at (7, 2.5) {$l_0$};
\node[] at (8.3, 2.5) {$l_1$};
\draw [black] plot [smooth] coordinates {(0, 0) (7, 4)};
\draw [black] plot [smooth] coordinates {(8, 0) (1.5, 4)};
\draw [black] plot [smooth] coordinates {(7, 4) (1.5, 4)};
\draw [black] plot [smooth] coordinates {(7, 4) (8, 0)};
\draw [black] plot [smooth] coordinates {(0, 0) (3.5, 1.5) (7, 4)};
\draw [black] plot [smooth] coordinates {(8, 0) (4, 2) (1.5, 4)};
\draw [black] plot [smooth] coordinates {(7, 4) (8, 2.5) (8, 0)};
\node[dot = green] at (0, 0) [label = below:$b$]{};
\node[dot = green] at (8, 0) [label = below:$B$]{};
\node[dot = red] at (1.5, 4) [label = above:$a$]{};
\node[dot = red] at (7, 4) [label = above:$A$]{};
    \end{tikzpicture}
        \caption{A visualization of the graph $G$.}
        \label{fig6}
    \end{figure}
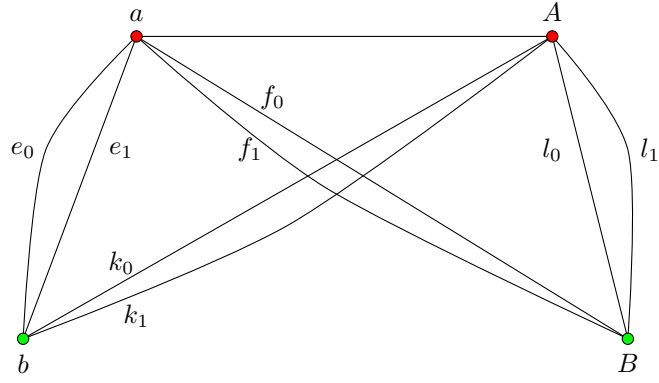
     Assume that $\mathcal{C}$ is a list of cycles of $G$, such that there are $c_0, c_1, c_2\in \mathbb{N}$ satisfying the following conditions:
    \begin{enumerate}
        \item Each edge appears in exactly $c_0$ cycles in $\mathcal{C}$.
        \item For each pair of edges $\{e, f\}\subseteq \delta(a)$ or $\{e, f\}\subseteq\delta(A)$, there are exactly $c_1$ cycles in $\mathcal{C}$ containing both $e$ and $f$.
        \item For each pair of edges $\{e, f\}\subseteq \delta(b)$ or $\{e, f\}\subseteq\delta(B)$, there are exactly $c_2$ cycles in $\mathcal{C}$ containing both $e$ and $f$.
    \end{enumerate}
    Observe that $\{e_0, e_1\}\subseteq \delta(a)\cap \delta(b)$. Thus $c_1$ and $c_2$ are both equal to the number of cycles containing both $e_0$ and $e_1$, in other words, $c_1=c_2 = c$. Furthermore, $\{e_0, e_1\}$ is a simple cycle, hence there are exactly $c$ copies of $\{e_0, e_1\}$ in $\mathcal{C}$. Consequently, $\mathcal{C}$ contains exactly $c$ copies of each bigon in $G$.\smallbreak
    Next, suppose that the cycle $\{e_0, k_1, aA\}$ appears $g_{0,1}$ times in $\mathcal{C}$, $\{e_1, k_1, aA\}$ appears $g_{1,1}$ times in $\mathcal{C}$, $\{e_0, k_0, aA\}$ appears $g_{0,0}$ times in $\mathcal{C}$, $\{e_1, k_0, aA\}$ appears $g_{1,0}$ times in $\mathcal{C}$. Since $\{e_0, k_1, aA\}$ and $\{e_1, k_1, aA\}$ are the only two cycles in $G$ containing both $aA$ and $k_1$, the number of cycles containing both $aA$ and $k_1$ is $(g_{0,1} + g_{1,1})$, or
    \begin{equation*}
        g_{0,1} + g_{1,1} = c.\tag{1}\label{eq:check_counterexample}
    \end{equation*}
    Similarly, we have $g_{0,0} + g_{1,0}=c$. Furthermore, $\{e_0, k_1, aA\}$ and $\{e_0, k_0, aA\}$ both contain the pair $\{e_0, aA\}$, thus $g_{0,1} + g_{0,0}$ is less than or equal to the number of cycles in $\mathcal{C}$ that pass through both $e_0$ and $aA$, which implies that $g_{0,1} + g_{0,0}\leq c$. Similarly, $g_{1,1} + g_{1,0}\leq c$. But from \eqref{eq:check_counterexample}, we have $g_{0,1} + g_{0,0} + g_{1,1} + g_{1,0} = 2c$, therefore, $g_{0,1} = g_{1,0}$ and $g_{0,0} = g_{1,1}$. Note that all the cycles containing $aA$ are triangles, consequently, by the above argument, the number of cycles containing $aA$ in $\mathcal{C}$ is $2\cdot 2c = 4c$, hence $c>0$. \smallbreak
    For each $(x, y, z, t)\in \{0, 1\}^4$, suppose that $g_{x, y, z, t}$ is the number of copies of $\{e_x, f_y, k_z, l_t\}$ in $\mathcal{C}$. Then, we have the following equations
    \begin{align*}
        g_{x, y, 0, 0} + g_{x, y, 1, 0} + g_{x, y, 0, 1} + g_{x, y, 1, 1} &= c\quad \forall x, y\in \{0, 1\}\\
        g_{x, 0, z, 0} + g_{x, 0, z, 1} + g_{x, 1, z, 0} + g_{x, 1, z, 1} &= c - g_{x, z}\quad \forall x, z\in \{0, 1\}.
    \end{align*}
    The first group of equations follows from considering the number of cycles that contain both $e_x$ and $f_y$, for each $(x, y)\in \{0, 1\}^2$, the second group of equations follows from considering the number of cycles that contain both $e_x$ and $k_z$, for each $(x, z)\in \{0, 1\}^2$ (for each such $(x, z)$, there are $c/2$ cycles of the form $\{e_x,k_z, aA\}$ already in $\mathcal{C}$). Summing up the first group of equations, we get
    \begin{equation*}
        \sum\limits_{(x, y, z, t)\in \{0, 1\}^4}g_{x, y, z, t} = 4c.
    \end{equation*}
    Summing up the second group of equations, we get
    \begin{equation*}
        \sum\limits_{(x, y, z, t)\in \{0, 1\}^4}g_{x, y, z, t} = 2c.
    \end{equation*}
    We noted earlier that $c>0$, so this is a contradiction.
\end{note}
\subsection{Computational evidences}
We computationally verify \cref{stronger} for six-vertex Whitehead graphs of lists of single word of length eleven. We give pseudocodes describing an algorithm of finding uniform balanced lists of cycles of Whitehead graphs. The main idea is to convert the problem of finding uniform balanced lists of cycles into integer linear programs (ILPs), see \cref{alg:find_uniform_balanced}. After that, we use Gurobi \cite{gurobi} to solve those ILPs. The actual Python codes and the explicit solutions of all cases are written in the ancillary files, with the same format as the lists of cycles written in \cref{main_proof_section}. Additionally, one can verify whether a graph is connected and pairwise well-connected by looking at its incidence matrix and adjoint matrix; in particular, a graph with $e$ edges is connected if and only if the rank of the incidence matrix is $e-1$; and the Preflow-Push algorithm (implemented in Networkx \cite{hagberg2008exploring}) is used to find the maximum flow between two vertices in a pair, from which we can check if a graph is pairwise well-connected.  
\begin{algorithm}[H]
\caption{Find a uniform balanced list of cycles of Whitehead graphs having $2n$ vertices and $l$ edges, for $n> 1$, and $l\in\mathbb{N}$.}\label{alg:find_uniform_balanced}
\begin{algorithmic}[1]
    \Require $n \in \mathbb{N}\setminus \{1\}$ \Comment{Rank of the free group.}
    \Require $l\in \mathbb{N}$ \Comment{Number of edges.}
    \State $S \gets \text{Set of all cyclically reduced words of length $l$ in $F_n$}$
    \State $solution\_dict \gets \text{empty dictionary}$ \Comment{The dictionary of solutions.}
    \State $counterexamples \gets \text{empty list}$ \Comment{The list of counterexamples (if any) of \cref{stronger}.}
    \For{$w \in S$}
    \State $U \gets (w)$ \Comment{The list containing a single word $w$.}
    \State $W(U) \gets \text{Whitehead graph of the list $U$}$
    \If{$W(U)$ is connected and pairwise well-connected}
    \State $M\gets $ The incidence matrix of $W(U)$ 
    \State $A \gets$ Set of all minimal nonzero incidence vectors $v\in \{0,1\}^{l}$ such that $M\cdot v\in \{0,2\}^{l}$ 
    \State\Comment{Each vector in $A$ corresponds to a cycle in $W(U)$.}
    \State $k\gets |A|$\Comment{Suppose that $A = \{v_1, \ldots, v_k\}$, where $\{v_1, \ldots, v_{m}\}$ is the set of cycles of length at least three.}
    \State \textbf{solve} the following integer linear program $(\mathcal{P})$:
    \begin{align*}
        \text{MINIMIZE}\hspace{0.5cm} & \sum\limits_{i=1}^{k}x_i\\
        \text{SUBJECT TO}& \\
        &x_1, \ldots, x_{k+1}\in\mathbb{Z}\\
        &x_i\ge 0, \forall i\in \{1, \ldots, k\}\\
        & x_{k+1} \ge 1\\
        &\sum\limits_{i=1}^{m}x_i \ge 1\tag{1}\\
        &\sum\limits_{i=1}^{k}x_i \cdot v_i = x_{k+1}\cdot (1,\ldots, 1)^T\tag{2}\\
        \text{\textbf{for} all pairs of adjacent edges $e, f\in \delta(v^{-1})$ \textbf{do}}\\
        & A(e, f) = \text{Set of all cycles in $A$ containing $e$ and $f$}\\
        & A(e^v, f^v) = \text{Set of all cycles in $A$ containing $e^{v}$ and $f^v$}\\
        & \sum\limits_{\substack{i\in\{1, \ldots, k\}\\v_i \in A(e,f)}}x_i\cdot v_i = \sum\limits_{\substack{i\in\{1, \ldots, k\}\\v_i \in A(e^v,f^v)}}x_i\cdot v_i\tag{3}\\
        \text{\textbf{end for}}&
    \end{align*}
    \State \textbf{end solve}
    \State \Comment{Explanation: $x_1, \ldots, x_k$ are number of occurrences of $v_1, \ldots, v_k$.}
    \State \Comment{Inequality (1) means that at least one cycle of length at least three is used.}
    \State \Comment{Equality (2) means that each edge appears in $x_{k+1}$ cycles.}
    \State\Comment{Equality (3) means that the number of cycles containing both $e$ and $f$ is equal to the number of cycles containing both $e^v$ and $f^v$.}
    \If{$(\mathcal{P})$ has an optimal solution $(x_1, \ldots, x_{k+1})$}
    \State $\mathcal{C}\gets$ For each $i\in \{1, \ldots, k\}$, $v_i$ occurs $x_i$ times \Comment{$\mathcal{C}$ is the uniform balanced list of cycles.}
    \State $solution\_dict[w] = \mathcal{C}$
    \Else
    \State \textbf{add} $w$ \textbf{to} $counterexamples$
    \EndIf
    \Else
    \State \textbf{continue}
    \EndIf
    \EndFor
    \State \Return $solution\_dict$, $counterexamples$ \Comment{\cref{stronger} is false if $counterexamples$ is non-empty.}
\end{algorithmic}
\end{algorithm}
\appendix
\section{An explicit surface subgroup}
Consider the word $w_0 = acAbAcabCb$, which corresponds to Case 1a in the proof of \cref{6_10_case}. We provide an explicit surface subgroup of the Baumslag double $D(U)$, where $U$ is the list containing only $w_0$. The group $D(U)$ has a presentation
\begin{equation*}
    D(U) = \langle a, b, c, d, e, f \mid aca^{-1}ba^{-1}cabc^{-1}b=dfd^{-1}ed^{-1}fdef^{-1}e\rangle.
\end{equation*}
Following the proof of \cite[Lemma 10]{kim2010}, the process starts by describing a $U$-polygonal surface from a balanced list of cycles in $W(U)$. See \cref{fig7} for an illustration of $W(U)$.
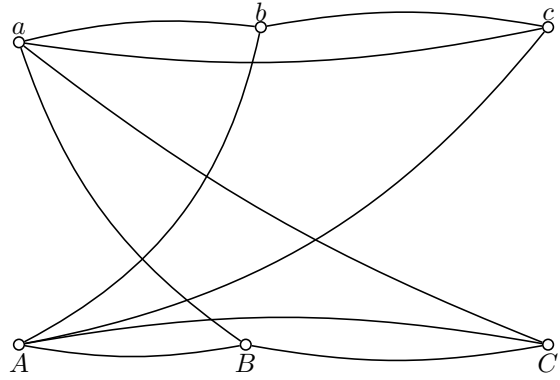
\begin{figure}[H]
        \centering
\begin{tikzpicture}[
    every node/.style={
        circle,
        draw,
        fill=white,
        minimum size=4pt,
        inner sep=1.0pt
    },
    line width=0.6pt
]

\node (A) at (0,0) {};
\node (B) at (3,0) {};
\node (C) at (7,0) {};

\node (a) at (0,4) {};
\node (b) at (3.2,4.2) {};
\node (c) at (7,4.2) {};

\node[draw=none, fill=none, above] at (a) {$a$};
\node[draw=none,fill=none,above] at (b) {$b$};
\node[draw=none,fill=none,above] at (c) {$c$};
\node[draw=none,fill=none,below] at (A) {$A$};
\node[draw=none,fill=none,below] at (B) {$B$};
\node[draw=none,fill=none,below] at (C) {$C$};

\draw (a) to[bend left=10] (b);
\draw (b) to[bend left=10] (c);
\draw (a) to[bend right=10] (c);

\draw (A) to[bend right=10] (B);
\draw (B) to[bend right=10] (C);
\draw (A) to[bend left=10] (C);

\draw (a) to[bend right=18] (B);

\draw (a) to[bend right=8] (C);

\draw (b) to[bend left=25] (A);


\draw (c) to[bend left=20] (A);

\end{tikzpicture}
\caption{The Whitehead graph of $U = [w_0]$.}
        \label{fig7}
    \end{figure}
By \eqref{eq:case1a}, a balanced list of cycles in $W(U)$ is $\mathcal{C} = [C_1, C_2, C_3, C_4, C_5, C_6, C_7, C_8, C_9, C_{10}, C_{11}, C_{12}]$, where 
\begin{align*}
 &C_1 = a\to b\to c\to A\to C\to B\to a,   \\
 &C_2 = a\to c\to A\to C\to B\to a,\\
 &C_3 = C_4 = a\to B\to A\to C\to a,\\
 &C_5=C_6 = a\to C\to B\to A\to c\to b\to a,\\
 &C_7 = b\to A\to C\to B\to a\to b,\\
 &C_8 = C_9 = b\to A\to c\to a\to b,\\
 &C_{10} = b\to A\to B\to a\to c\to b,\\
 &C_{11} = b\to A\to C\to a\to c\to b,\\
 &C_{12} = b\to A\to B\to C\to a\to c\to b.
\end{align*}
We choose a total order $\prec$ on $\{(v, e)\in V(W(U))\times E(W(U))\mid e\in \delta(v)\}$ as follows.
\begin{align*}
    (a, aC)\prec (a, ac)\prec (a, ab)\prec (a, aB)\prec \\
    (A, Ab)\prec (A, AB)\prec (A, AC)\prec (A, Ac)\prec\\
    (b, bA)\prec (b, bc)\prec (b, ba)\prec\\
    (B, BC)\prec (B, Ba)\prec (B, BA)\prec\\
    (c, ca)\prec (c, cA)\prec (c, cb)\prec\\
    (C, Ca)\prec (C, CA)\prec (C, CB).
\end{align*}
Next, for $i\in \{1, \ldots, 12\}$, let $V_i$ be a polygonal disk such that $\partial V_i$ and $C_i$ are cycles having the same length, such that the vertices of $V_i$ are labeled by the edges of $C_i$, and there is a directed edge from vertex labeled $f$ to vertex labeled $e$ if $e$ and $f$ are adjacent to a vertex $x$ in $C_i$, and $(x, e)\prec (x, f)$. We also label such a directed edge $\Vec{fe}$ by $(x, \{e, f\})$. 
\setcounter{equation}{0}
Now, define an edge-pairing $\sim_0$ on $V_{1}, \ldots, V_{12}$ as follows. See \cref{fig8} for an illustration of the cycles $V_1, \ldots, V_{12}$ and the edge-pairing.\\
\begin{align}
    (a, \{ab, aB\})\in E(\partial V_{1})\sim_0 (A, \{AC, Ac\})\in E(\partial V_{2}),\\
    (a, \{ac, aB\})\in E(\partial V_{2})\sim_0 (A, \{AB, Ac\})\in E(\partial V_{5}),\\ 
    (a, \{aC, aB\})\in E(\partial V_{3})\sim_0 (A, \{Ab, Ac\})\in E(\partial V_{8}),\\ 
    (a, \{aC, aB\})\in E(\partial V_{4})\sim_0 (A, \{Ab, Ac\})\in E(\partial V_{9}),\\ 
    (a, \{aC, ab\})\in E(\partial V_{5})\sim_0 (A, \{Ab, AC\})\in E(\partial V_{7}),\\ 
    (a, \{aC, ab\})\in E(\partial V_{6})\sim_0 (A, \{Ab, AC\})\in E(\partial V_{11}),\\
    (a, \{ab, aB\})\in E(\partial V_{7})\sim_0 (A, \{AC, Ac\})\in E(\partial V_{1}),\\ 
    (a, \{ac, ab\})\in E(\partial V_{8})\sim_0 (A, \{AB, AC\})\in E(\partial V_{3}),\\
    (a, \{ac, ab\})\in E(\partial V_{9})\sim_0 (A, \{AB, AC\})\in E(\partial V_{4}),\\
    (a, \{ac, aB\})\in E(\partial V_{10})\sim_0 (A, \{AB, Ac\})\in E(\partial V_{6}),\\ 
    (a, \{aC, ac\})\in E(\partial V_{11})\sim_0 (A, \{Ab, AB\})\in E(\partial V_{12}),\\ 
    (a, \{aC, ac\})\in E(\partial V_{12})\sim_0 (A, \{Ab, AB\})\in E(\partial V_{10}), \\
    (b, \{bc, ba\})\in E(\partial V_{1})\sim_0 (B, \{Ba, BA\})\in E(\partial V_{10}),\\
    (b, \{bc, ba\})\in E(\partial V_{5})\sim_0 (B, \{Ba, BA\})\in E(\partial V_{3}),\\
    (b, \{bc, ba\})\in E(\partial V_{6})\sim_0 (B, \{Ba, BA\})\in E(\partial V_{4}),\\
    (b, \{bA, ba\})\in E(\partial V_{7})\sim_0 (B, \{BC, BA\})\in E(\partial V_{12}),\\
    (b, \{bA, ba\})\in E(\partial V_{8})\sim_0 (B, \{BC, BA\})\in E(\partial V_{5}),\\
    (b, \{bA, ba\})\in E(\partial V_{9})\sim_0 (B, \{BC, BA\})\in E(\partial V_{6}),\\
    (b, \{bA, bc\})\in E(\partial V_{10})\sim_0 (B, \{BC, Ba\})\in E(\partial V_{1}),\\ 
    (b, \{bA, bc\})\in E(\partial V_{11})\sim_0 (B, \{BC, Ba\})\in E(\partial V_{2}),\\
    (b, \{bA, bc\})\in E(\partial V_{12})\sim_0 (B, \{BC, Ba\})\in E(\partial V_{7}),\\
    (c, \{cA, cb\})\in E(\partial V_{1})\sim_0 (C, \{CA, CB\})\in E(\partial V_{2}),\\
    (c, \{ca, cA\})\in E(\partial V_{2})\sim_0 (C, \{Ca, CA\})\in E(\partial V_{11}),\\ 
    (c, \{cA, cb\})\in E(\partial V_{5})\sim_0 (C, \{CA, CB\})\in E(\partial V_{1}),\\ 
    (c, \{cA, cb\})\in E(\partial V_{6})\sim_0 (C, \{CA, CB\})\in E(\partial V_{7}),\\ 
    (c, \{ca, cA\})\in E(\partial V_{8})\sim_0 (C, \{Ca, CA\})\in E(\partial V_{3}),\\ 
    (c, \{ca, cA\})\in E(\partial V_{9})\sim_0 (C, \{Ca, CA\})\in E(\partial V_{4}),\\ 
    (c, \{ca, cb\})\in E(\partial V_{10})\sim_0 (C, \{Ca, CB\})\in E(\partial V_{12}),\\ 
    (c, \{ca, cb\})\in E(\partial V_{11})\sim_0 (C, \{Ca, CB\})\in E(\partial V_{6}),\\ 
    (c, \{ca, cb\})\in E(\partial V_{12})\sim_0 (C, \{Ca, CB\})\in E(\partial V_{5}).
\end{align}
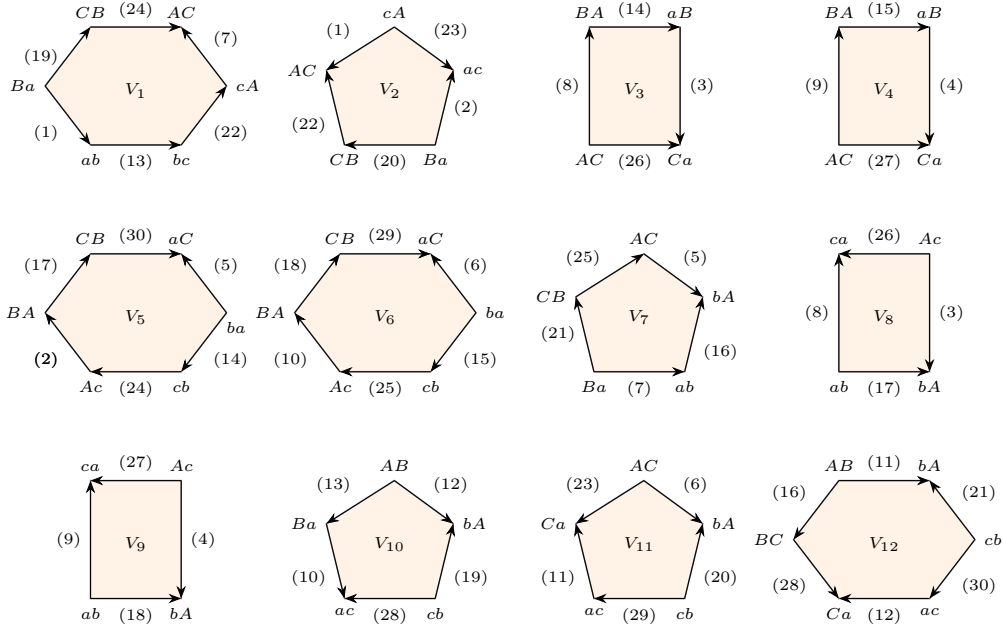
\begin{figure}[H]
    \centering
    \begin{tikzpicture}[scale=0.6,
    >=Stealth,
    vertex/.style={
        fill=orange!10,
        draw=black!45,
        line width=.6pt
    },
    every node/.style={
        font=\tiny
    }
]

\begin{scope}[shift={(0,0)}]

\coordinate (A) at (0,1.3);
\coordinate (B) at (2,1.3);
\coordinate (C) at (3,0);
\coordinate (D) at (2,-1.3);
\coordinate (E) at (0,-1.3);
\coordinate (F) at (-1,0);

\filldraw[vertex]
(A)--(B)--(C)--(D)--(E)--(F)--cycle;

\draw[->] (A)--(B) node[midway, above]{(24)};
\draw[->] (C)--(B) node[midway, above right]{(7)};
\draw[->] (D)--(C) node[midway, below right]{(22)};
\draw[->] (E)--(D) node[midway, below ]{(13)};
\draw[->] (F)--(E) node[midway, below left]{(1)};
\draw[->] (F)--(A) node[midway,  left]{(19)};

\node[above] at (A) {$CB$};
\node[above] at (B) {$AC$};
\node[right] at (C) {$cA$};
\node[below] at (D) {$bc$};
\node[below] at (E) {$ab$};
\node[left]  at (F) {$Ba$};

\node at (1,-0.1) {$V_1$};

\end{scope}

\begin{scope}[shift={(5.5,0)}]

\coordinate (A) at (1.2,1.3);
\coordinate (B) at (2.5,.35);
\coordinate (C) at (2.1,-1.3);
\coordinate (D) at (0.1,-1.3);
\coordinate (E) at (-.3,.35);

\filldraw[vertex]
(A)--(B)--(C)--(D)--(E)--cycle;

\draw[->] (A)--(B) node[midway, above right]{(23)};
\draw[->] (C)--(B) node[midway, right]{(2)};
\draw[->] (C)--(D) node[midway, below ]{(20)};
\draw[->] (D)--(E) node[midway, below left]{(22)};
\draw[->] (A)--(E) node[midway, above left]{(1)};

\node[above] at (A) {$cA$};
\node[right] at (B) {$ac$};
\node[below] at (C) {$Ba$};
\node[below] at (D) {$CB$};
\node[left]  at (E) {$AC$};

\node at (1.1,-0.1) {$V_2$};

\end{scope}

\begin{scope}[shift={(11,0)}]

\coordinate (A) at (0,1.3);
\coordinate (B) at (2,1.3);
\coordinate (C) at (2,-1.3);
\coordinate (D) at (0,-1.3);

\filldraw[vertex]
(A)--(B)--(C)--(D)--cycle;

\draw[->] (A)--(B) node[midway, above]{(14)};
\draw[->] (B)--(C) node[midway, right]{(3)};
\draw[->] (D)--(C) node[midway, below]{(26)};
\draw[->] (D)--(A) node[midway, left]{(8)};

\node[above] at (A) {$BA$};
\node[above] at (B) {$aB$};
\node[below] at (C) {$Ca$};
\node[below] at (D) {$AC$};

\node at (1,-0.1) {$V_3$};

\end{scope}
\begin{scope}[shift={(16.5,0)}]

\coordinate (A) at (0,1.3);
\coordinate (B) at (2,1.3);
\coordinate (C) at (2,-1.3);
\coordinate (D) at (0,-1.3);

\filldraw[vertex]
(A)--(B)--(C)--(D)--cycle;

\draw[->] (A)--(B) node[midway, above]{(15)};
\draw[->] (B)--(C) node[midway, right]{(4)};
\draw[->] (D)--(C) node[midway, below]{(27)};
\draw[->] (D)--(A) node[midway, left]{(9)};

\node[above] at (A) {$BA$};
\node[above] at (B) {$aB$};
\node[below] at (C) {$Ca$};
\node[below] at (D) {$AC$};

\node at (1,-0.1) {$V_4$};

\end{scope}

\begin{scope}[shift={(0,-5)}]

\coordinate (A) at (0,1.3);
\coordinate (B) at (2,1.3);
\coordinate (C) at (3,0);
\coordinate (D) at (2,-1.3);
\coordinate (E) at (0,-1.3);
\coordinate (F) at (-1,0);

\filldraw[vertex]
(A)--(B)--(C)--(D)--(E)--(F)--cycle;

\draw[->] (A)--(B) node[midway, above]{(30)};
\draw[->] (C)--(B) node[midway, above right]{(5)};
\draw[->] (C)--(D) node[midway, below right]{(14)};
\draw[->] (D)--(E) node[midway, below ]{(24)};
\draw[->] (E)--(F) node[midway, below left]{(2)} node[midway, below left]{(2)};
\draw[->] (F)--(A) node[midway, above left]{(17)};

\node[above] at (A) {$CB$};
\node[above] at (B) {$aC$};
\node[below] at (C) {\hspace{1em}$ba$};
\node[below] at (D) {$cb$};
\node[below] at (E) {$Ac$};
\node[left]  at (F) {$BA$};

\node at (1,-0.1) {$V_5$};

\end{scope}
\begin{scope}[shift={(5.5,-5)}]

\coordinate (A) at (0,1.3);
\coordinate (B) at (2,1.3);
\coordinate (C) at (3,0);
\coordinate (D) at (2,-1.3);
\coordinate (E) at (0,-1.3);
\coordinate (F) at (-1,0);

\filldraw[vertex]
(A)--(B)--(C)--(D)--(E)--(F)--cycle;

\draw[->] (A)--(B) node[midway, above]{(29)};
\draw[->] (C)--(B) node[midway, above right]{(6)};
\draw[->] (C)--(D) node[midway, below right]{(15)};
\draw[->] (D)--(E) node[midway, below ]{(25)};
\draw[->] (E)--(F) node[midway, below left]{(10)};
\draw[->] (F)--(A) node[midway, above left]{(18)};

\node[above] at (A) {$CB$};
\node[above] at (B) {$aC$};
\node[right] at (C) {$ba$};
\node[below] at (D) {$cb$};
\node[below] at (E) {$Ac$};
\node[left]  at (F) {$BA$};

\node at (1,-0.1) {$V_6$};

\end{scope}

\begin{scope}[shift={(11,-5)}]

\coordinate (A) at (1.2,1.3);
\coordinate (B) at (2.5,.35);
\coordinate (C) at (2.1,-1.3);
\coordinate (D) at (0.1,-1.3);
\coordinate (E) at (-.3,.35);

\filldraw[vertex]
(A)--(B)--(C)--(D)--(E)--cycle;

\draw[->] (A)--(B) node[midway, above right]{(5)};
\draw[->] (C)--(B) node[midway, below right]{(16)};
\draw[->] (D)--(C) node[midway, below ]{(7)};
\draw[->] (D)--(E) node[midway, left]{(21)};
\draw[->] (E)--(A) node[midway, above left]{(25)};

\node[above] at (A) {$AC$};
\node[right] at (B) {$bA$};
\node[below] at (C) {$ab$};
\node[below] at (D) {$Ba$};
\node[left]  at (E) {$CB$};

\node at (1.1,-0.1) {$V_7$};

\end{scope}

\begin{scope}[shift={(16.5,-5)}]

\coordinate (A) at (0,1.3);
\coordinate (B) at (2,1.3);
\coordinate (C) at (2,-1.3);
\coordinate (D) at (0,-1.3);

\filldraw[vertex]
(A)--(B)--(C)--(D)--cycle;

\draw[->] (B)--(A) node[midway, above]{(26)};
\draw[->] (D)--(A) node[midway,  left]{(8)};
\draw[->] (D)--(C) node[midway, below]{(17)};
\draw[->] (B)--(C) node[midway, right]{(3)};

\node[above] at (A) {$ca$};
\node[above] at (B) {$Ac$};
\node[below] at (C) {$bA$};
\node[below] at (D) {$ab$};

\node at (1,-0.1) {$V_8$};

\end{scope}

\begin{scope}[shift={(0,-10)}]

\coordinate (A) at (0,1.3);
\coordinate (B) at (2,1.3);
\coordinate (C) at (2,-1.3);
\coordinate (D) at (0,-1.3);

\filldraw[vertex]
(A)--(B)--(C)--(D)--cycle;

\draw[->] (B)--(A) node[midway, above]{(27)};
\draw[->] (D)--(A) node[midway,  left]{(9)};
\draw[->] (D)--(C) node[midway, below ]{(18)};
\draw[->] (B)--(C) node[midway, right]{(4)};

\node[above] at (A) {$ca$};
\node[above] at (B) {$Ac$};
\node[below] at (C) {$bA$};
\node[below] at (D) {$ab$};

\node at (1,-0.1) {$V_9$};

\end{scope}

\begin{scope}[shift={(5.5,-10)}]

\coordinate (A) at (1.2,1.3);
\coordinate (B) at (2.5,.35);
\coordinate (C) at (2.1,-1.3);
\coordinate (D) at (0.1,-1.3);
\coordinate (E) at (-.3,.35);

\filldraw[vertex]
(A)--(B)--(C)--(D)--(E)--cycle;

\draw[->] (A)--(B) node[midway, above right]{(12)};
\draw[->] (C)--(B) node[midway, below right]{(19)};
\draw[->] (C)--(D) node[midway, below ]{(28)};
\draw[->] (E)--(D) node[midway, below left]{(10)};
\draw[->] (A)--(E) node[midway, above left]{(13)};

\node[above] at (A) {$AB$};
\node[right] at (B) {$bA$};
\node[below] at (C) {$cb$};
\node[below] at (D) {$ac$};
\node[left]  at (E) {$Ba$};

\node at (1.1,-0.1) {$V_{10}$};

\end{scope}

\begin{scope}[shift={(11,-10)}]

\coordinate (A) at (1.2,1.3);
\coordinate (B) at (2.5,.35);
\coordinate (C) at (2.1,-1.3);
\coordinate (D) at (0.1,-1.3);
\coordinate (E) at (-.3,.35);

\filldraw[vertex]
(A)--(B)--(C)--(D)--(E)--cycle;

\draw[->] (A)--(B) node[midway, above right]{(6)};
\draw[->] (C)--(B) node[midway, below right]{(20)};
\draw[->] (C)--(D) node[midway, below ]{(29)};
\draw[->] (D)--(E) node[midway, below left]{(11)};
\draw[->] (A)--(E) node[midway, above left]{(23)};

\node[above] at (A) {$AC$};
\node[right] at (B) {$bA$};
\node[below] at (C) {$cb$};
\node[below] at (D) {$ac$};
\node[left]  at (E) {$Ca$};

\node at (1.1,-0.1) {$V_{11}$};

\end{scope}

\begin{scope}[shift={(16.5,-10)}]

\coordinate (A) at (0,1.3);
\coordinate (B) at (2,1.3);
\coordinate (C) at (3,0);
\coordinate (D) at (2,-1.3);
\coordinate (E) at (0,-1.3);
\coordinate (F) at (-1,0);

\filldraw[vertex]
(A)--(B)--(C)--(D)--(E)--(F)--cycle;

\draw[->] (A)--(B) node[midway, above]{(11)};
\draw[->] (C)--(B) node[midway, above right]{(21)};
\draw[->] (C)--(D) node[midway, below right]{(30)};
\draw[->] (D)--(E) node[midway, below ]{(12)};
\draw[->] (F)--(E) node[midway, below left]{(28)};
\draw[->] (A)--(F) node[midway, above left]{(16)};

\node[above] at (A) {$AB$};
\node[above] at (B) {$bA$};
\node[right] at (C) {$cb$};
\node[below] at (D) {$ac$};
\node[below] at (E) {$Ca$};
\node[left]  at (F) {$BC$};

\node at (1,-0.1) {$V_{12}$};

\end{scope}
\end{tikzpicture}
    \caption{The cycles $V_1, \ldots, V_{12}$ and the edge-pairing.}
    \label{fig8}
\end{figure}
Let $S = \bigsqcup\limits_{i\in \{1, \ldots, 12\}}V_i/\sim_0$, on which we have a graph 
\begin{equation*}
    \Gamma=\bigsqcup\limits_{i\in \{1, \ldots, 12\}}\partial V_i/\sim_0.
\end{equation*}
Note that in the process of identifying directed edges, we are also identifying the corresponding vertices of the edges in consideration. In particular, consider the vertex labeled $ab$ of $\partial V_1$. 
\begin{enumerate}[1.]
    \item After identification (1), vertex $ab\in V(\partial V_1)$ is glued with $AC\in V(\partial V_2)$.
    \item After identification (22), vertex $AC\in V(\partial V_2)$ is glued with $cA\in V(\partial V_1)$.
    \item After identification (7), vertex $cA\in V(\partial V_1)$ is glued with $aB\in V(\partial V_7)$.
    \item After identification (21), vertex $Ba\in V(\partial V_7)$ is glued with $bc\in V(\partial V_{12})$.
    \item After identification (30), vertex $cb\in V(\partial V_{12})$ is glued with $CB\in V(\partial V_5)$.
    \item After identification (17), vertex $BC\in V(\partial V_5)$ is glued with $bA\in V(\partial V_8)$.
    \item After identification (3), vertex $bA\in V(\partial V_8)$ is glued with $aC\in V(\partial V_3)$.
    \item After identification (26), vertex $aC\in V(\partial V_3)$ is glued with $ca\in V(\partial V_8)$.
    \item After identification (8), vertex $ca\in V(\partial V_8)$ is glued with $AB\in V(\partial V_3)$.
    \item After identification (14), vertex $AB\in V(\partial V_3)$ is glued with $ba\in V(\partial V_5)$.
    \item After identification (5), vertex $ab\in V(\partial V_5)$ is glued with $CA\in V(\partial V_7)$.
    \item After identification (25), vertex $AC\in V(\partial V_7)$ is glued with $cA\in V(\partial V_6)$.
    \item After identification (10), vertex $Ac\in V(\partial V_6)$ is glued with $aB\in V(\partial V_{10})$.
    \item After identification (13), vertex $aB\in V(\partial V_{10})$ is glued with $bc\in V(\partial V_{1})$.
    \item After identification (22), vertex $cb\in V(\partial V_1)$ is glued with $CB\in V(\partial V_2)$.
    \item After identification (20), vertex $BC\in V(\partial V_2)$ is glued with $bA\in V(\partial V_{11})$.
    \item After identification (6), vertex $Ab\in V(\partial V_{11})$ is glued with $aC\in V(\partial V_6)$.
    \item After identification (29), vertex $aC\in V(\partial V_6)$ is glued with $ca\in V(\partial V_{11})$.
    \item After identification (11), vertex $ca\in V(\partial V_{11})$ is glued with $AB\in V(\partial V_{12})$.
    \item After identification (16), vertex $AB\in V(\partial V_{12})$ is glued with $ab\in V(\partial V_7)$.
    \item After identification (7), vertex $ab\in V(\partial V_7)$ is glued with $AC\in V(\partial V_1)$.
    \item After identification (24), vertex $AC\in V(\partial V_1)$ is glued with $cA\in V(\partial V_5)$.
    \item After identification (2), vertex $cA\in V(\partial V_5)$ is glued with $aB\in V(\partial V_2)$.
    \item After identification (20), vertex $aB\in V(\partial V_2)$ is glued with $bc\in V(\partial V_{11})$.
    \item After identification (29), vertex $cb\in V(\partial V_{11})$ is glued with $CB\in V(\partial V_6)$.
    \item After identification (18), vertex $BC\in V(\partial V_6)$ is glued with $bA\in V(\partial V_9)$.
    \item After identification (4), vertex $bA\in V(\partial V_9)$ is glued with $aC\in V(\partial V_4)$.
    \item After identification (27), vertex $aC\in V(\partial V_4)$ is glued with $ca\in V(\partial V_9)$.
    \item After identification (9), vertex $ac\in V(\partial V_9)$ is glued with $AB\in V(\partial V_4)$.
    \item After identification (15), vertex $AB\in V(\partial V_4)$ is glued with $ba\in V(\partial V_6)$.
    \item After identification (6), vertex $ab\in V(\partial V_6)$ is glued with $AC\in V(\partial V_{11})$.
    \item After identification (23), vertex $AC\in V(\partial V_{11})$ is glued with $cA\in V(\partial V_2)$.
    \item After identification (1), vertex $cA\in V(\partial V_2)$ is glued with $aB\in V(\partial V_1)$.
    \item After identification (19), vertex $Ba\in V(\partial V_1)$ is glued with $bc\in V(\partial V_{10})$.
    \item After identification (28), vertex $cb\in V(\partial V_{10})$ is glued with $CB\in V(\partial V_{12})$.
    \item After identification (16), vertex $BC\in V(\partial V_{12})$ is glued with $bA\in V(\partial V_7)$.
    \item After identification (5), vertex $bA\in V(\partial V_7)$ is glued with $aC\in V(\partial V_5)$.
    \item After identification (30), vertex $aC\in V(\partial V_5)$ is glued with $ca\in V(\partial V_{12})$.
    \item After identification (12), vertex $ca\in V(\partial V_{12})$ is glued with $AB\in V(\partial V_{10})$.
    \item After identification (13), vertex $AB\in V(\partial V_{10})$ is glued with $ba\in V(\partial V_1)$.
\end{enumerate}
Therefore, the $40$ vertices shown above are glued together, which becomes a vertex $u_1$ of $\Gamma$. Next, consider the vertex $ac\in V(\partial V_2)$, we have
\begin{enumerate}[1.]
    \item After identification (2), vertex $ac\in V(\partial V_2)$ is glued with $BA\in V(\partial V_5)$.
    \item After identification (17), vertex $BA\in V(\partial V_5)$ is glued with $ba\in V(\partial V_8)$.
    \item After identification (8), vertex $ab\in V(\partial V_8)$ is glued with $AC\in V(\partial V_3)$.
    \item After identification (26), vertex $AC\in V(\partial V_3)$ is glued with $cA\in V(\partial V_8)$.
    \item After identification (3), vertex $cA\in V(\partial V_8)$ is glued with $aB\in V(\partial V_3)$.
    \item After identification (14), vertex $Ba\in V(\partial V_3)$ is glued with $bc\in V(\partial V_5)$.
    \item After identification (24), vertex $cb\in V(\partial V_5)$ is glued with $CB\in V(\partial V_1)$.
    \item After identification (19), vertex $BC\in V(\partial V_1)$ is glued with $bA\in V(\partial V_{10})$.
    \item After identification (12), vertex $Ab\in V(\partial V_{10})$ is glued with $aC\in V(\partial V_{12})$.
    \item After identification (28), vertex $Ca\in V(\partial V_{12})$ is glued with $ca\in V(\partial V_{10})$.
    \item After identification (10), vertex $ac\in V(\partial V_{10})$ is glued with $BA\in V(\partial V_6)$.
    \item After identification (18), vertex $BA\in V(\partial V_6)$ is glued with $ba\in V(\partial V_9)$.
    \item After identification (9), vertex $ab\in V(\partial V_9)$ is glued with $AC\in V(\partial V_4)$.
    \item After identification (27), vertex $AC\in V(\partial V_4)$ is glued with $cA\in V(\partial V_9)$.
    \item After identification (4), vertex $cA\in V(\partial V_9)$ is glued with $aB\in V(\partial V_4)$.
    \item After identification (15), vertex $Ba\in V(\partial V_4)$ is glued with $bc\in V(\partial V_6)$.
    \item After identification (25), vertex $cb\in V(\partial V_6)$ is glued with $CB\in V(\partial V_7)$.
    \item After identification (21), vertex $BC\in V(\partial V_7)$ is glued with $bA\in V(\partial V_{12})$.
    \item After identification (11), vertex $Ab\in V(\partial V_{12})$ is glued with $aC\in V(\partial V_{11})$.
    \item After identification (23), vertex $Ca\in V(\partial V_{11})$ is glued with $ca\in V(\partial V_{2})$.
\end{enumerate}
Therefore, the $20$ vertices shown above are glued together, which becomes a vertex $u_2$ of $\Gamma$. Consequently, $\Gamma$ has $2$ vertices, $30$ edges, and $S\backslash \Gamma$ has $12$ connected components homeomorphic to $\mathbb{R}^2$, corresponding to the interiors of $V_i (i\in \{1, \ldots, 12\})$. In particular, $\chi(S) = 2-30+12=-16$. Note that the dual graph $\Gamma^*$ of $\Gamma$ also embeds in $S$, providing a $2$-dimensional cell complex structure on $S$. The $1$-skeleton of $S$ is $\Gamma^*$, and the $2$-cells of $S$ are the connected components of $S\backslash \Gamma^*$. By definition, $\Gamma^*$ has $12$ vertices, $30$ edges, and $S\backslash \Gamma^*$ has $2$ connected components $P_1$, $P_2$, which are homeomorphic to $\mathbb{R}^2$. The boundary of $P_1$ is a cycle, corresponding to the link of $u_1$ in $\Gamma$, simlilarly, the boundary of $P_2$ corresponds to the link of $u_2$ in $\Gamma$. Note that, there is an immersion $\phi:\Gamma^*\to Cay(F_3)/F_3$ induced by the orientation and labels of the edges of $\Gamma$, and we need to show that the maps 
\begin{align*}
    \partial P_1\to \Gamma^*\xrightarrow{\phi}Cay(F_3)/F_3\\
    \partial P_2\to \Gamma^*\xrightarrow{\phi}Cay(F_3)/F_3
\end{align*}
read two cyclic conjugates of powers of $w_0$. Concretely, for $k\in \{1, 2\}$, each directed edge $\overrightarrow{V_i V_j}$ of $\partial P_k$ corresponds to an edge $e$ of $\Gamma$, which is a part of the border between two disks $V_i$ and $V_j$, such that the label of $e$ in $\partial V_i$ is of the form $(x, \{u_1, u_2\})$, where $x\in \{A, B, C\}$ and $u_1, u_2\in \delta(x)$. In such case, the edge $\overrightarrow{V_i V_j}$ is labeled by $x^{-1}\in \{a, b, c\}$. We consecutively read the label of edges in $\partial P_1$ as follows.
\begin{itemize}
    \item The edge $e_1=\overrightarrow{V_2 V_1}$ corresponding to $(a, \{ab, aB\})\in E(\partial V_{1})\sim_0 (A, \{AC, Ac\})\in E(\partial V_{2})$ has label $a$.
    \item Consider the corner of the vertex $AC=(ab)^A$ in $\partial V_2$, we see that the successor of $e_1$ is $e_2 = \overrightarrow{V_2 V_1}$ corresponding to $(C, \{CA, CB\})\in E(\partial V_2)\sim_0 (c, \{cA, cb\})\in E(\partial V_1)$, which has label $c$.
    \item Consider the corner of the vertex $cA = (AC)^c$ in $\partial V_1$, we see that the successor of $e_2$ is $e_3 = \overrightarrow{V_1 V_7}$ corresponding to $(A, \{AC, Ac\})\in E(\partial V_1)\sim_0 (a, \{ab, aB\})\in E(\partial V_7)$, which has label $a$.
    \item Consider the corner of the vertex $aB = (cA)^a$ in $\partial V_7$, we see that the successor of $e_3$ is $e_4 = \overrightarrow{V_{12} V_7}$ corresponding to $(B, \{BC, Ba\})\in E(\partial V_7)\sim_0 (b, \{bA, bc\})\in E(\partial V_{12})$, which has label $b$.
    \item Consider the corner of the vertex $bc = (Ba)^b$ in $\partial V_{12}$, we see that the successor of $e_4$ is $e_5 = \overrightarrow{V_5 V_{12}}$ corresponding to $(c, \{ca, cb\})\in E(\partial V_{12})\sim_0 (C, \{Ca, CB\})\in E(\partial V_5)$, which has label $c$.
    \item Consider the corner of the vertex $CB = (bc)^C$ in $\partial V_5$, we see that the successor of $e_5$ is $e_6 = \overrightarrow{V_5 V_8}$ corresponding to $(B, \{BC, BA\})\in E(\partial V_5)\sim_0 (b, \{bA, ba\})\in E(\partial V_8)$, which has label $b$.
    \item Consider the corner of the vertex $bA = (CB)^b$ in $\partial V_8$, we see that the successor of $e_6$ is $e_7 = \overrightarrow{V_8 V_3}$ corresponding to $(A, \{Ab, Ac\})\in E(\partial V_8)\sim_0 (a, \{aC, aB\})\in E(\partial V_3)$, which has label $a$.
    \item Consider the corner of the vertex $aC = (Ab)^a$ in $\partial V_3$, we see that the successor of $e_7$ is $e_8 = \overrightarrow{V_3 V_8}$ corresponding to $(C, \{Ca, CA\})\in E(\partial V_3)\sim_0 (c, \{ca, cA\})\in E(\partial V_8)$, which has label $c$.
    \item Consider the corner of the vertex $ca = (aC)^c$ in $\partial V_8$, we see that the successor of $e_8$ is $e_9 = \overrightarrow{V_3 V_8}$ corresponding to $(a, \{ac, ab\})\in E(\partial V_8)\sim_0 (A, \{AB, AC\})\in E(\partial V_3)$, which has label $a$.
    \item Consider the corner of the vertex $AB = (ca)^A$ in $\partial V_3$, we see that the successor of $e_9$ is $e_{10} = \overrightarrow{V_3 V_5}$ corresponding to $(B, \{Ba, BA\})\in E(\partial V_3)\sim_0 (b, \{bc, ba\})\in E(\partial V_5)$, which has label $b$.
    \item Consider the corner of the vertex $ba = (AB)^b$ in $\partial V_5$, we see that the successor of $e_{10}$ is $e_{11} = \overrightarrow{V_7 V_5}$ corresponding to $(a, \{aC, ab\})\in E(\partial V_5)\sim_0 (A, \{Ab, AC\})\in E(\partial V_7)$, which has label $a$.
    \item Consider the corner of the vertex $AC = (ba)^A$ in $\partial V_7$, we see that the successor of $e_{11}$ is $e_{12} = \overrightarrow{V_7 V_6}$ corresponding to $(C, \{CA, CB\})\in E(\partial V_7)\sim_0 (c, \{cA, cb\})\in E(\partial V_6)$, which has label $c$.
    \item Consider the corner of the vertex $cA = (AC)^c$ in $\partial V_6$, we see that the successor of $e_{12}$ is $e_{13} = \overrightarrow{V_6 V_{10}}$ corresponding to $(A, \{AB, Ac\})\in E(\partial V_6)\sim_0 (a, \{ac, aB\})\in E(\partial V_{10})$, which has label $a$.
    \item Consider the corner of the vertex $aB = (cA)^a$ in $\partial V_{10}$, we see that the successor of $e_{13}$ is $e_{14} = \overrightarrow{V_{10} V_1}$ corresponding to $(B, \{Ba, BA\})\in E(\partial V_{10})\sim_0 (b, \{bc, ba\})\in E(\partial V_1)$, which has label $b$.
    \item Consider the corner of the vertex $bc = (aB)^b$ in $\partial V_1$, we see that the successor of $e_{14}$ is $e_{15} = \overrightarrow{V_2 V_1}$ corresponding to $(c, \{cA, cb\})\in E(\partial V_1)\sim_0 (C, \{CA, CB\})\in E(\partial V_2)$, which has label $c$.
\end{itemize}
Continuing by the same fashion, we arrive at the following sequence of consecutive edges in $\partial P_1$
\begin{align*}
    &e_{16}=(B, \{BC, Ba\})\in E(\partial V_2)\sim_0 (b, \{bA, bc\})\in E(\partial V_{11}), (BC)^b=bA,\\
    &e_{17}=(A, \{Ab, AC\})\in E(\partial V_{11})\sim_0 (a, \{aC, ab\})\in E(\partial V_{6}), (bA)^a = aC,\\ 
    &e_{18}=(C, \{Ca, CB\})\in E(\partial V_6)\sim_0 (c, \{ca, cb\})\in E(\partial V_{11}),(aC)^c = ca,\\ 
    &e_{19}=(a, \{aC, ac\})\in E(\partial V_{11})\sim_0 (A, \{Ab, AB\})\in E(\partial V_{12}),(ca)^A = AB,\\ 
    &e_{20}=(B, \{BC, BA\})\in E(\partial V_{12})\sim_0 (b, \{bA, ba\})\in E(\partial V_{7}),(AB)^b = ba,\\
    &e_{21}=(a, \{ab, aB\})\in E(\partial V_7)\sim_0 (A, \{AC, Ac\})\in E(\partial V_{1}), (ab)^A=AC,\\
    &e_{22}=(C, \{CA, CB\})\in E(\partial V_{1})\sim_0 (c, \{cA, cb\})\in E(\partial V_{5}), (AC)^c = cA,\\ 
    &e_{23}=(A, \{AB, Ac\})\in E(\partial V_5)\sim_0 (a, \{ac, aB\})\in E(\partial V_{2}),(Ac)^a = aB,\\ 
    &e_{24}=(B, \{BC, Ba\})\in E(\partial V_{2})\sim_0 (b, \{bA, bc\})\in E(\partial V_{11}),(aB)^b = bc,\\ 
    &e_{25}=(c, \{ca, cb\})\in E(\partial V_{11})\sim_0 (C, \{Ca, CB\})\in E(\partial V_{6}),(bc)^C = CB,\\
    &e_{26}=(B, \{BC, BA\})\in E(\partial V_6)\sim_0 (b, \{bA, ba\})\in E(\partial V_{9}), (BC)^b=bA,\\
    &e_{27}=(A, \{Ab, Ac\})\in E(\partial V_{9})\sim_0 (a, \{aC, aB\})\in E(\partial V_{4}), (bA)^a = aC,\\ 
    &e_{28}=(C, \{Ca, CA\})\in E(\partial V_4)\sim_0 (c, \{ca, cA\})\in E(\partial V_{9}),(aC)^c = ca,\\ 
    &e_{29}=(a, \{ac, ab\})\in E(\partial V_{9})\sim_0 (A, \{AB, AC\})\in E(\partial V_{4}),(ca)^A = AB,\\ 
    &e_{30}=(B, \{Ba, BA\})\in E(\partial V_{4})\sim_0 (b, \{bc, ba\})\in E(\partial V_{6}),(AB)^b = ba,\\
    &e_{31}=(a, \{aC, ab\})\in E(\partial V_6)\sim_0 (A, \{Ab, AC\})\in E(\partial V_{11}), (ba)^A=AC,\\
    &e_{32}=(C, \{Ca, CA\})\in E(\partial V_{11})\sim_0 (c, \{ca, cA\})\in E(\partial V_{2}), (AC)^c = cA,\\ 
    &e_{33}=(A, \{AC, Ac\})\in E(\partial V_2)\sim_0 (a, \{ab, aB\})\in E(\partial V_{1}),(Ac)^a = aB,\\ 
    &e_{34}=(B, \{BC, Ba\})\in E(\partial V_{1})\sim_0 (b, \{bA, bc\})\in E(\partial V_{10}),(aB)^b = bc,\\ 
    &e_{35}=(c, \{ca, cb\})\in E(\partial V_{10})\sim_0 (C, \{Ca, CB\})\in E(\partial V_{12}),(bc)^C = CB,\\
    &e_{36}=(B, \{BC, BA\})\in E(\partial V_{12})\sim_0 (b, \{bA, ba\})\in E(\partial V_{7}), (BC)^b=bA,\\
    &e_{37}=(A, \{Ab, AC\})\in E(\partial V_{7})\sim_0 (a, \{aC, ab\})\in E(\partial V_{5}), (bA)^a = aC,\\ 
    &e_{38}=(C, \{Ca, CB\})\in E(\partial V_5)\sim_0 (c, \{ca, cb\})\in E(\partial V_{12}),(aC)^c = ca,\\ 
    &e_{39}=(a, \{aC, ac\})\in E(\partial V_{12})\sim_0 (A, \{Ab, AB\})\in E(\partial V_{10}),(ca)^A = AB,\\ 
    &e_{40}=(B, \{Ba, BA\})\in E(\partial V_{10})\sim_0 (b, \{bc, ba\})\in E(\partial V_{1}),(AB)^b = ba,\\
    &e_{41}=(a, \{ab, aB\})\in E(\partial V_1)\sim_0 (A, \{AC, Ac\})\in E(\partial V_{2}) = e_1.
\end{align*}
We conclude that $\partial P_1 = e_1 e_2 \cdots e_{40}$ is a cycle of length $40$, and it reads $(AcabCbacAb)^4$, which is a cyclic conjugate of a power of $w_0$, as stated. Next, we read the label of edges in $\partial P_2$.
\begin{align*}
     &f_{1}=(a, \{ac, aB\})\in E(\partial V_2)\sim_0 (A, \{AB, Ac\})\in E(\partial V_{5}), (ca)^A=AB,\\
     &f_{2}=(B, \{BC, BA\})\in E(\partial V_5)\sim_0 (b, \{bA, ba\})\in E(\partial V_{8}), (AB)^b=ba,\\
     &f_{3}=(a, \{ac, ab\})\in E(\partial V_8)\sim_0 (A, \{AB, AC\})\in E(\partial V_{3}), (ab)^A=AC,\\
     &f_{4}=(C, \{Ca, CA\})\in E(\partial V_3)\sim_0 (c, \{ca, cA\})\in E(\partial V_{8}), (CA)^c=cA,\\
     &f_{5}=(A, \{Ab, Ac\})\in E(\partial V_8)\sim_0 (a, \{aC, aB\})\in E(\partial V_{3}), (Ac)^a=aB,\\
     &f_{6}=(B, \{Ba, BA\})\in E(\partial V_3)\sim_0 (b, \{bc, ba\})\in E(\partial V_{5}), (aB)^b=bc,\\
     &f_{7}=(c, \{cA, cb\})\in E(\partial V_5)\sim_0 (C, \{CA, CB\})\in E(\partial V_{1}), (cb)^C=CB,\\
     &f_{8}=(B, \{BC, Ba\})\in E(\partial V_1)\sim_0 (b, \{bA, bc\})\in E(\partial V_{10}), (BC)^b=bA,\\
     &f_{9}=(A, \{Ab, AB\})\in E(\partial V_{10})\sim_0 (a, \{aC, ac\})\in E(\partial V_{12}), (bA)^a=aC,\\
     &f_{10}=(C, \{Ca, CB\})\in E(\partial V_{12})\sim_0 (c, \{ca, cb\})\in E(\partial V_{10}), (Ca)^c=ca,\\
     &f_{11}=(a, \{ac, aB\})\in E(\partial V_{10})\sim_0 (A, \{AB, Ac\})\in E(\partial V_{6}), (ca)^A=AB,\\
     &f_{12}=(B, \{BC, BA\})\in E(\partial V_6)\sim_0 (b, \{bA, ba\})\in E(\partial V_{9}), (AB)^b=ba,\\
     &f_{13}=(a, \{ac, ab\})\in E(\partial V_9)\sim_0 (A, \{AB, AC\})\in E(\partial V_{4}), (ab)^A=AC,\\
     &f_{14}=(C, \{Ca, CA\})\in E(\partial V_4)\sim_0 (c, \{ca, cA\})\in E(\partial V_{9}), (CA)^c=cA,\\
     &f_{15}=(A, \{Ab, Ac\})\in E(\partial V_9)\sim_0 (a, \{aC, aB\})\in E(\partial V_{4}), (Ac)^a=aB,\\
     &f_{16}=(B, \{Ba, BA\})\in E(\partial V_{4})\sim_0 (b, \{bc, ba\})\in E(\partial V_{6}), (aB)^b=bc,\\
     &f_{17}=(c, \{cA, cb\})\in E(\partial V_6)\sim_0 (C, \{CA, CB\})\in E(\partial V_{7}), (cb)^C=CB,\\
     &f_{18}=(B, \{BC, Ba\})\in E(\partial V_7)\sim_0 (b, \{bA, bc\})\in E(\partial V_{12}), (BC)^b=bA,\\
     &f_{19}=(A, \{Ab, AB\})\in E(\partial V_{12})\sim_0 (a, \{aC, ac\})\in E(\partial V_{11}), (bA)^a=aC,\\
     &f_{20}=(C, \{Ca, CA\})\in E(\partial V_{11})\sim_0 (c, \{ca, cA\})\in E(\partial V_{2}), (Ca)^c=ca,\\
     &f_{21} = (a, \{ac, aB\})\in E(\partial V_2)\sim_0 (A, \{AB, Ac\})\in E(\partial V_{5}) = f_1.
\end{align*}
We conclude that $\partial P_2 = f_1 f_2 \cdots f_{20}$ is a cycle of length $20$, and it reads $(AbAcabCbac)^2$, which is a cyclic conjugate of a power of $w_0$, as stated.

Now, remove an open ball $B_1$ inside $P_1$ and an open ball $B_2$ inside $P_2$, because $\partial P_1$ and $\partial P_2$ read powers of cyclic conjugates of $w_0$, from \cite[Lemma 5, Theorem 6]{kim2010geometricity}, we conclude that the double of $S\backslash (B_1\cup B_2)$ admits a $\pi_1$-injective embedding into a finite covering space of $X(U)$. As a cell complex, the Euler characteristic of the double of $S\backslash (B_1\cup B_2)$ is $2(\chi(S)-2) = -36$, which implies that the fundamental group of the double of $S\backslash B_1$ is a hyperbolic surface subgroup of $D(U)$.

Concretely, we find a presentation for the fundamental group of the double of $S\backslash (B_1\cup B_2)$ as follows. The graph $\Gamma^*$ has $12$ vertices labeled $V_1, \ldots, V_{12}$, and has $30$ edges labeled $(1), \ldots, (30)$ as given above. In particular, $\Gamma^*$ is connected, and a spanning tree of $\Gamma^*$ is given by the edges 
\begin{equation*}
    (1), (2), (3), (4), (5), (6), (10), (11), (13), (14), (15).
\end{equation*}
Therefore, the fundamental group $\pi_1(\Gamma^*, V_1)$ is a free group of rank $19$, freely generated by the following cycles
\begin{align*}
    x_1 = &\textbf{(7)}(5)(2)(1)\\
    x_2 = &(1)(2)(14)(3)\textbf{(8)}(14)(2)(1)\\
    x_3 = &(13)(10)(15)\textbf{(9)}(4)(15)(10)(13)\\
    x_4 = &(13)\textbf{(12)}(11)(6)(10)(13)\\
    x_5 = &(1)(2)(5)\textbf{(16)}(11)(6)(10)(13)\\
    x_6 = &(1)(2)(14)(3)\textbf{(17)}(2)(1)\\
    x_7 = &(13)(10)\textbf{(18)}(4)(15)(10)(13)\\
    x_8 = &(13)\textbf{(19)}\\
    x_9 = &(1)\textbf{(20)}(6)(10)(13)\\
    x_{10} = &(1)(2)(5)\textbf{(21)}(11)(6)(10)(13)\\
    x_{11} = &(1)\textbf{(22)}\\
    x_{12} = &(1)\textbf{(23)}(6)(10)(13)\\
    x_{13} = &(1)(2)\textbf{(24)}\\
    x_{14} = &(13)(10)\textbf{(25)}(5)(2)(1)\\
    x_{15} = &(1)(2)(14)(3)\textbf{(26)}(14)(2)(1)\\
    x_{16} = &(13)(10)(15)\textbf{(27)}(4)(15)(10)(13)\\
    x_{17} = &(13)\textbf{(28)}(11)(6)(10)(13)\\
    x_{18} = &(13)(10)(6)\textbf{(29)}(10)(13)\\
    x_{19} = &(13)(10)(6)(11)\textbf{(30)}(2)(1).
\end{align*}
Consider the map $\phi: \Gamma^* \looparrowright Cay(F_3)/F_3$ such that the image of each directed edge of the form $(v^{-1}, \{e, f\})\sim_0 (v, \{e^v, f^v\})$ for $v\in \{A, B, C\}$ is the loop $v\in Cay(F_3)/F_3$. In this case, let $\phi_*: \pi_1(\Gamma^*, V_1)\to \pi_1(Cay(F_3)/F_3) = \langle a, b, c\rangle$, we have
\begin{align*}
    \phi(x_1) &=aaaa\\
    \phi(x_2)&=AABAAbaa\\
    \phi(x_3)&=BABAAbab\\
    \phi(x_4)&=Baaaab\\
    \phi(x_5)&=AAABaaab\\
    \phi(x_6)&=AABABaa\\
    \phi(x_7)&=BAbabab\\
    \phi(x_8)&=BB\\
    \phi(x_9)&=Abaab\\
    \phi(x_{10})&=AAAbaaab\\
    \phi(x_{11})&=Ac\\
    \phi(x_{12})&=ACaab\\
    \phi(x_{13})&=AAC\\
    \phi(x_{14})&=BACaaa\\
    \phi(x_{15})&=AABACbaa\\
    \phi(x_{16})&=BABcabab\\
    \phi(x_{17})&=BCaaab\\
    \phi(x_{18})&=BAACab\\
    \phi(x_{19})&=BAAACaa.
\end{align*}
By definition, the double of $S\backslash(B_1\cup B_2)$ has a presentation
\begin{equation*}
    G=\langle x_1, \ldots, x_{19}, y_1, \ldots, y_{19}, t_2\mid u_1=v_1, t_2 u_2 t_2^{-1}=v_2\rangle,
\end{equation*}
where $u_1$ is the word corresponding to the cycle $e_1 e_2\cdots e_{40} = \partial P_1$, and $v_1$ is obtained from $u_1$ by replacing $x_i$'s with $y_i$'s, similarly, $u_2$ is the word corresponding to the cycle $f_1 f_2\cdots f_{20} = \partial P_2$, and $v_2$ is obtained from $u_2$ by replacing $x_i$'s with $y_i$'s. From the description of the cycles, we have
\begin{align*}
    u_1&=x_{11}x_1 x_{10}x_{19}x_6^{-1}x_{15}^{-1}x_2 x_{14}x_{11}^{-1}x_9 x_{18}^{-1}x_5^{-1}x_1^{-1}x_{13}^{-1}x_9 x_{18}x_7 x_{16}x_3 x_{12}^{-1}x_8^{-1}x_{17}x_5^{-1}x_{19}^{-1}x_4^{-1},\\
    v_1 &= y_{11}y_1 y_{10}y_{19}y_6^{-1}y_{15}^{-1}y_2 y_{14}y_{11}^{-1}y_9 y_{18}^{-1}y_5^{-1}y_1^{-1}y_{13}^{-1}y_9 y_{18}y_7 y_{16}y_3 y_{12}^{-1}y_8^{-1}y_{17}y_5^{-1}y_{19}^{-1}y_4^{-1},\\
    u_2&=x_{6}^{-1}x_2 x_{15}^{-1}x_{13}x_8^{-1}x_4 x_{17}^{-1}x_7 x_3 x_{16} x_{14}x_{10}x_{12}^{-1},\\
    v_2&=y_{6}^{-1}y_2 y_{15}^{-1}y_{13}y_8^{-1}y_4 y_{17}^{-1}y_7 y_3 y_{16} y_{14}y_{10}y_{12}^{-1}.
\end{align*}
Let $\psi: G\to D(U)$ be the homomorphism of fundamental groups induced from an embedding of the double of $S\backslash (B_1\cup B_2)$ into a finite cover of $X(U)$, which comes from the labels of edges of $\Gamma^*$, then $\psi(x_i) = \phi(x_i)$ as given above, and $\psi(y_i)$ is obtained from $\psi(x_i)$ by replacing $(a, b, c)$ with $(d, e, f)$, respectively, and by the definition of $t_2$, we have
\begin{equation*}
    \psi(t_2) = A(AbAcabCbac)^2 a.
\end{equation*}
Therefore, the image of $G$ is the subgroup generated by $\{\psi(x_1), \ldots, \psi(x_{19}), \psi(y_1), \ldots, \psi(y_{19}), \phi(t_2)\}$, which is a surface subgroup of $D(U)$.

\bibliographystyle{plain} 
\bibliography{citations} 
\end{document}